\documentclass{monsky2009}

\usepackage{amssymb}

\usepackage{bm}

\newenvironment{conjecture*}[1][]{\textbf{Conjecture #1\hspace{.3em}}}{}
\newenvironment{theorem*}[1]{\textbf{#1}\itshape \hspace{.3em}}{\upshape}
\newenvironment{remark*}[1]{\textbf{#1}\itshape \hspace{.3em}}{\upshape}
\newenvironment{example*}[1]{\textbf{#1}\itshape \hspace{.3em}}{\upshape}
\newenvironment{proof}[1][]{\textbf{Proof #1\hspace{.3em}}}{}

\newenvironment{proofsketch}{\textbf{Proof Sketch\hspace{.3em}}}{}

\newtheorem{definition}{Definition}[section]
\newtheorem{theorem}[definition]{Theorem}
\newtheorem{lemma}[definition]{Lemma}

\newtheorem{corollary}[definition]{Corollary}

\newcounter{kpremark}
\newcounter{proofitem}
\newenvironment{proofresult}[1]{\addtocounter{proofitem}{1}
\leftskip.25in
\rightskip\leftskip
(\theproofitem)\quad #1}

\newcommand{\mod}[1]{\ensuremath{\hspace{.5em}(#1)}}

\newcommand{\newo}{\ensuremath{\mathcal{O}}}

\newcommand{\bbc}{\ensuremath{\mathbb{C}}}

\newcommand{\me}{\mathrm{e}} 

\newcommand{\modd}{\ensuremath{M(\mathit{odd})}}

\begin{document}

\begin{frontmatter}






\title{A Hecke algebra attached to mod 2 modular forms of level 5}
\author{Paul Monsky}

\address{Brandeis University, Waltham MA  02454-9110, USA. monsky@brandeis.edu}

\begin{abstract}
Let $F$ be the element $\sum_{n\ \mathit{odd},\ n>0}x^{n^{2}}$ of $Z/2[[x]]$. Set $G=F(x^{5})$, $D=F(x)+F(x^{25})$. For $k>0$, $(k,10)=1$, define $D_{k}$ as follows. $D_{1}=D$, $D_{3}=D^{8}/G$, $D_{7}=D^{2}G$, $D_{9}=D^{4}G$; furthermore $D_{k+10}=G^{2}D_{k}$.

Using modular forms of level $\Gamma_{0}(5)$ we show that the space $W$ spanned by the $D_{k}$ is stabilized by the formal Hecke operators $T_{p}:Z/2[[x]]\rightarrow Z/2[[x]]$, $p\ne 2$ or $5$. And we determine the structure of the (completed) shallow Hecke algebra attached to $W$. This algebra proves to be a power series ring in $T_{3}$ and $T_{7}$ with an element of square $0$ adjoined. As Hecke module, $W$ identifies with a certain subquotient of the space of mod~2 modular forms of level $\Gamma_{0}(5)$, and our Hecke algebra result parallels findings in level 1 (by J.-L. Nicolas and J.-P. Serre) and in level $\Gamma_{0}(3)$ by us.
\end{abstract}


\end{frontmatter}


\section{Some spaces of mod 2 modular forms of level \bm{$\Gamma_{0}(5)$}}
\label{section1}

Nicolas and Serre \cite{3}, \cite{4} have proved various results about the action of the Hecke algebra on the space of mod~2 modular forms of level 1. In \cite{1} we gave a variant of their results in level $\Gamma_{0}(3)$.  Here we find close analogues in level $\Gamma_{0}(5)$.

We first summarize results from \cite{3}, \cite{4} just as we did at the start of \cite{1}.  There are commuting formal Hecke operations $T_{p}:Z/2[[x]]\rightarrow Z/2[[x]]$, one for each odd prime $p$. Here $T_{p}(\sum c_{n}x^{n}) = \sum c_{pn}x^{n} + \sum c_{n}x^{pn}$. Let $F$ in $Z/2[[x]]$ be $\sum_{n \mathit{odd},\ n>0}x^{n^{2}}$. Using modular forms of level 1, Nicolas and Serre show that the $T_{p}$ stabilize the space spanned by $F, F^{3}, F^{5}, F^{7}, \ldots $ and that the associated (completed) Hecke algebra is a power series ring in $T_{3}$ and $T_{5}$. Indeed they make the space into a faithful $Z/2[[X,Y]]$-module with $X$ and $Y$ acting by $T_{3}$ and $T_{5}$, and show that each $T_{p}$ is multiplication by an element of the maximal ideal $(X,Y)$.

In \cite{1} we took $D$ in $Z/2[[x]]$ to be $\sum_{(n,6)=1,\ n>0}x^{n^{2}}$, so that $D=F(x)+F(x^{9})$. $W1$ had a basis consisting of the $D^{k}$ with $k\equiv 1\mod{6}$ and $W5$ a basis consisting of the $D^{k}$, $k\equiv 5\mod{6}$; $W$ was $W1\oplus W5$. We showed that the $T_{p}$ with $p\equiv 1\mod{6}$ stabilize $W1$ and $W5$, while when $p\equiv 5\mod{6}$, $T_{p}(W1)\subset W5$ and $T_{p}(W5)\subset W1$.  We further showed that $W1$ has a basis $m_{i,j}$ ``adapted to $T_{7}$ and $T_{13}$'' with $m_{0,0}=D$. We deduced that the (completed shallow) Hecke algebra attached to $W$ is a power series ring in $T_{7}$ and $T_{13}$ with an element of square $0$ adjoined. Though $W$ is a space of mod~2 modular forms of level $\Gamma_{0}(9)$, it identifies as Hecke-module with a certain subquotient of the space of odd mod~2 modular forms of level $\Gamma_{0}(3)$.

In the present work we change notation. Now $D=\sum_{(n,10)=1,\ n>0}x^{n^{2}}$, so that $D=F(x)+F(x^{25})$. Let $G=F(x^{5})$. For $k>0$ with $(k,10)=1$ define $D_{k}$ as follows: $D_{1}=D$, $D_{3}=D^{8}/G$, $D_{7}=D^{2}G$ and $D_{9}=D^{4}G$, while $D_{k+10}=G^{2}D_{k}$. Let $W$ be spanned by the $D_{k}$. (As $D_{k}=x^{k}+\cdots $, these are linearly independent over $Z/2$.) Then $W=W_{a}\oplus W_{b}$ where the $D_{k}$, $k\equiv 1, 3, 7, 9\mod{20}$, are a basis of $W_{a}$, and the $D_{k}$, $k\equiv 11, 13, 17, 19\mod{20}$ are a basis of $W_{b}$. We will establish the following analogues to the results of \cite{1}.

\begin{enumerate}
\item[(1)]
The $T_{p}$, $p\ne 5$, stabilize $W$. If $p\equiv 1, 3, 7, 9\mod{20}$, $T_{p}$ stabilizes $W_{a}$ and $W_{b}$, while if $p\equiv 11, 13, 17, 19\mod{20}$, $T_{p}(W_{a})\subset W_{b}$ and $T_{p}(W_{b})\subset W_{a}$.
\item[(2)]
Though $W$ is a space of mod~2 modular forms of level $\Gamma_{0}(25)$, it identifies as Hecke-module with a certain subquotient of the space of odd mod~2 modular forms of level $\Gamma_{0}(5)$.
\item[(3)]
$W_{a}$ admits a basis $m_{i,j}$ adapted to $T_{3}$ and $T_{7}$ in the sense of Nicolas and Serre, with $m_{0,0}=D$. 
\item[(4)]
The (completed shallow) Hecke algebra attached to $W$ is a power series ring in $T_{3}$ and $T_{7}$ with an element of square $0$ adjoined.
\end{enumerate}

The proofs of (1) and (2) occupy sections \ref{section1} and \ref{section2}. Some ingredients are the mod~2 level~5 modular equation $(F+G)^{6}=FG$ of Theorem \ref{theorem1.12} and the relation $D^{15}+G^{4}D^{3}+G^{3}= 0$ of Lemma \ref{lemma2.4}.

The proofs of (3) and (4) resemble those of corresponding results in \cite{1}. We identify $W_{a}$ with a  subspace $V^{\prime}$ of the polynomial ring $Z/2[w]$, making $D_{k}$ correspond to $w^{k}$, and show that under this identification $T_{3}(D_{k})$ corresponds to a certain $P_{k}$ defined in \cite{2}, with $P_{k+80}=w^{80}P_{k}+w^{20}P_{k+20}$. Lemma 5.5 of \cite{2}, which deals with this recursion, gives insight into the action of $T_{3}$. Suppose in particular that $q$ is a power of 2, and let $W_{a}(q)$ be spanned by the $D_{k}$ in $W_{a}$ with $k<40q^{2}$. Using this Lemma 5.5 we find that the kernel of $T_{3}:W_{a}(q)\rightarrow W_{a}(q)$ has dimension at most $2q$. Ideal theory in $Z[\sqrt{-10}]$, developed in section \ref{section4}, then shows that this kernel is a space $DI(q)$ of theta-series attached to binary quadratic forms. A study of the action of $T_{7}$ on $DI(q)$, together with formalism from \cite{1}, leads, in section \ref{section5}, to a proof of (3). And the further study in section \ref{section6} of $T_{11}:W_{a}\rightarrow W_{b}$ and $W_{b}\rightarrow W_{a}$ gives a proof of (4).

We begin our proofs by introducing elements $P$,$E_{4}$, and $B$ of $\bbc [[x]]$. These are the expansions at infinity of classical modular forms of level $\Gamma_{0}(5)$.

\begin{definition}
\label{def1.1}
$P=1+\cdots$, $E_{4}=1+\cdots$ and $B=x+\cdots$ are the expansions at infinity of:
\vspace{-2ex}
\begin{enumerate}
\item[(1)] The normalized weight~2 Eisenstein series for $\Gamma_{0}(5)$.
\item[(2)] The normalized weight~4 Eisenstein series of level~1.
\item[(3)] The normalized weight~4 Eisenstein series, vanishing at infinity, for $\Gamma_{0}(5)$.
\end{enumerate}
\end{definition}

\begin{remark*}{Remark}
One usually views expansions at infinity of modular forms as elements of $\bbc [[q]]$ where $q=\me^{2\pi iz}$. But as in \cite{1} we'll use the letter $x$ rather than the letter $q$.
\end{remark*}

\begin{definition}
\label{def1.2}
$r$ in $Z/2[[x]]$ is $\sum_{n>0}(x^{n^{2}}+x^{2n^{2}}+x^{5n^{2}}+x^{10n^{2}})$.
\end{definition}

Classical formul\ae\ for the coefficients of Eisenstein series show that $P$, $E_{4}$ and $B$ are in $Z[[x]]$ with mod~2 reductions $1$, $1$ and $r$.

\begin{definition}
\label{def1.3}
$C$ in $\bbc[[x]]$ is the expansion at infinity of the normalized weight~4 cusp form $(\eta(z)\eta(5z))^{4}$ for  $\Gamma_{0}(5)$.
\end{definition}

Now the expansion of $\eta(z)$ at infinity is $x^{1/24}$(an element $1-x-x^{2}+\cdots$ of $Z[[x]]$). We deduce:

\begin{lemma}
\label{lemma1.4}
$C$ is in $Z[[x]]$ and is $x-4x^{2}+\cdots$.
\end{lemma}

We now show that the mod~2 reduction $\bar{C}$ of $C$ is $r^{2}+r$.

\begin{lemma}
\label{lemma1.5}
Let $n$ be a positive integer. The number of $(a,b)$ in $Z\times Z$ with $a^{2}+5b^{2}=6n$ and $a\equiv b\mod{3}$ is $2$~mod~$4$ if $n$ is either a square or $5$(square), and $0$~mod~$4$ otherwise.
\end{lemma}

\begin{proof}
Let $S(n)$ be the set of such $(a,b)$. Then $T:(a,b)\rightarrow (-a,-b)$ and $U:(a,b)\rightarrow \left(\frac{2a-5b}{3},\frac{-a-2b}{3}\right)$ are commuting involutions of $S(a)$. $T$ has no fixed points, while the fixed points of $U\ (\mbox{resp.}\ TU)$ are of the form $(5k,-k)\ (\mbox{resp.}\ (k,k))$. In the first case, $5k^{2}=n$ while in the second $k^{2}=n$. So we have an action of $Z/2\times Z/2$ on $S(n)$ in which all orbits are of size 4 with the following exceptions. When $n=5k^{2}$ there is a size 2 orbit $\pm (5k,-k)$. Where $n=k^{2}$ there is a size 2 orbit $\pm (k,k)$. The result follows.
\qed
\end{proof}

\begin{lemma}
\label{lemma1.6}
Let $n$ be an integer. The number of $(a,b)$ in $Z\times Z$ with $a^{2}+5b^{2}=6n$ and $a\equiv b\equiv 1\mod{6}$ is odd if $n$ is an odd square or $5\cdot$(an odd square), and even otherwise.
\end{lemma}

\begin{proof}
We may assume $n\ge 0$. Since $a^{2}+5b^{2}$ is $2$~mod~$4$, there are no such pairs when $n$ is even. For fixed odd $n>0$ the pairs of Lemma \ref{lemma1.5} come in three types according as $a\equiv b\equiv 1\mod{6}$, $a\equiv b\equiv 5\mod{6}$, or $a\equiv b\equiv 3\mod{6}$.  There are as many pairs of the first type as there are of the second. So in view of Lemma \ref{lemma1.5} it suffices to show that the number of pairs of the third type is a multiple of 4.  But these pairs come in sets of 4, $(\pm a, \pm b)$.
\qed
\end{proof}

\begin{lemma}
\label{lemma1.7}
The mod~2 reduction $\bar{C}$ of $C$ is $r^{2}+r$.
\end{lemma}

\begin{proof}
Using the familiar expansion of $\eta(z)$ at infinity we find that $\bar{C}=\left(\sum_{a\equiv 1\mod{6}}x^{a^{2}/6}\right)\left(\sum_{b\equiv 1\mod{6}}x^{5b^{2}/6}\right)$. The coefficient of $x^{n}$ in this element of $Z/2[[x]]$ is the mod~2 reduction of the number of $(a,b)$ in $Z\times Z$ with $a^{2}+5b^{2}=6n$ and $a\equiv b\equiv 1\mod{6}$. So by Lemma \ref{lemma1.6}, $\bar{C}=\sum_{n\ \mathit{odd},\ n>0}(x^{n^{2}}+x^{5n^{2}})$. This is precisely $r^{2}+r$.
\qed
\end{proof}

\begin{definition}
\label{def1.8}
If $k\ge 0$ and even, $M_{k}$ consists of those $f$ in $Z/2[[x]]$ for which there is a weight $k$ modular form of level $\Gamma_{0}(5)$ whose expansion at infinity lies in $Z[[x]]$ and reduces to $f$.
\end{definition}

Using multiplication by $E_{4}$ we see that $M_{0}\subset M_{4}\subset M_{8}\subset \ldots$

\begin{definition}
\label{def1.9}
$M=\cup M_{4m}$ is ``the space of mod~2 modular forms of level $\Gamma_{0}(5)$.'' $\modd$ consists of the odd elements of $M$, i.e.\ those elements lying in $x\cdot Z/2[[x^{2}]]$.
\end{definition}

Note that $M$ is a subring of $Z/2[[x]]$. Since the reductions of $E_{4}$ and $B$ ane $1$ and $r$, $M\supset Z/2[r]$/ Since $r^{2}+r=\sum_{n\ \mathit{odd},\ n>0}(x^{n^{2}}+x^{5n^{2}})$ is odd, each $r^{2k}(r^{2}+r)$ is in $\modd$.

\begin{theorem}
\label{theorem1.10}
Fix $m\ge 0$ and suppose $0\le i \le 2m$. Then there is a weight $4m$ modular form $u_{i}$ of level $\Gamma_{0}(5)$ whose expansion at infinity has the following properties. It is $x^{i}+\ldots$, lies in $Z[[x]]$, and reduces to $r^{i}$.
\end{theorem}

\begin{proof}
It suffices to prove this when $m=1$, so the weight is $4$. For $i=0$ we take $P^{2}$ whose expansion reduces to $1$, while for $i=1$ we take $B$ whose expansion reduces to $r$. Now $E_{4}=1+240x+2160x^{2}+\cdots$, $P^{2}=(1 +6x+18x^{2}+\cdots)^{2}$ and $B=x+9x^{2}+\cdots$ are expansions at infinity of weight $4$ forms of level $\Gamma_{0}(5)$, and $E_{4}-P^{2}-228B=36x^{2}+\cdots$. Also $B-C=13x^{2}+\cdots$, so a $Z$-linear combination of these last two expansions such as $4(\mathit{first})-11(\mathit{second})$ is $x^{2}+\cdots$. And the mod~2 reduction of this linear combination is $\bar{B}-\bar{C}$ which is $r^{2}$ by Lemma \ref{lemma1.7}.
\qed
\end{proof}

\begin{theorem}
\label{theorem1.11}
The $r^{i}$, $0\le i \le 2m$, are a basis of $M_{4m}$ over $Z/2$. It follows that $M=Z/2[r]$. Furthermore the $r^{2k}(r^{2}+r)$ are a basis of $\modd$ over $Z/2$.
\end{theorem}

\begin{proof}
Theorem \ref{theorem1.10} shows that the $r^{i}$, $0\le i \le 2m$, lie in $M_{4m}$. Now the $u_{i}$, $0\le i \le 2m$, of Theorem \ref{theorem1.10} are linearly independent over $\bbc$. Classical dimension formul\ae\ then show that over $\bbc$ they span the space of weight $4m$ modular forms of level $\Gamma_{0}(5)$. Suppose $u$ lies in this space and that the expansion of $u$ at infinity lies in $Z[[x]]$. Writing $u$ as a $\bbc$-linear combination of the $u_{i}$ and examining the expansions we see that $u$ is a $Z$-linear combination of the $u_{i}$. So the reduction of $u$ is a $Z/2$-linear combination of $r^{i}$, $0\le i \le 2m$, giving the first two results. It follows also that the $r^{2i}$ with $0\le i \le m$ together with the $r^{2i}(r^{2}+r)$ with $0\le i \le m-1$ are a $Z/2$-basis of $M_{4m}$. We conclude that the $r^{2i}(r^{2}+r)$ with $0\le i \le m-1$ are a basis for the subspace of $M_{4m}$ consisting of odd power series.
\qed
\end{proof}

Recall now that $F$ in $Z/2[[x]]$ is $\sum_{n\ \mathit{odd},\ n>0} x^{n^{2}}$, while $G=F(x^{5})$. We have seen that $r^{2}+r=F+G$.

\begin{theorem}
\label{theorem1.12}\hspace{2em}\\
\vspace{-4ex}
\begin{enumerate}
\item[(1)] $F$ and $G$ lie in $M_{12}$.
\item[(2)] $G=r^{5}(r+1)$, $F=r(r+1)^{5}$, and $(F+G)^{6}=FG$.
\end{enumerate}
\end{theorem}

\begin{proof}
Let $\Delta$ be Ramanujan's weight~12 level~1 cusp form. Then $F$ and $G$ are the reductions of the expansions at infinity of $\Delta(z)$ and $\Delta(5z)$, and so lie in $M_{12}$. Theorem \ref{theorem1.11} shows that $G$ is a $Z/2$-linear combination of $r^{2}+r$, $r^{4}+r^{3}$ and $r^{6}+r^{5}$. As $G=x^{5}+\cdots$ it can only be $r^{6}+r^{5}$. Then $F=r^{6}+r^{5}+r^{2}+r=r(r+1)^{5}$. The final result is immediate.
\qed
\end{proof}

\begin{theorem}
\label{theorem1.13}
$\modd$ is spanned by the $F^{i}G^{j}$ with $i+j$ odd.
\end{theorem}

\begin{proof}
$r^{2}+r=F+G$, while $r^{4}+r^{3}=(r^{2}+r)^{3}+r^{6}+r^{5}=(F+G)^{3}+G$. Furthermore, $(r^{4}+r^{2})(r^{2k+2}+r^{2k+1})+r^{2k+4}+r^{2k+3}=r^{2k+6}+r^{2k+5}$. Since $r^{4}+r^{2}=F^{2}+G^{2}$, an induction on $k$ shows that each $r^{2k+2}+r^{2k+1}$ is a sum of $F^{i}G^{j}$ with $i+j$ odd.
\qed
\end{proof}

Suppose now that $p$ is prime, $p\ne 2$ or $5$. Then we have commuting formal Hecke operators $T_{p}:Z/2[[x]]\rightarrow Z/2[[x]]$ with $T_{p}$ taking $\sum c_{n}x^{n}$ to $\sum c_{pn}x^{n}+\sum c_{n}x^{pn}$.

\begin{theorem}
\label{theorem1.14}
The $T_{p}$ stabilize $M$ and $\modd$. In fact they stabilize the space spanned by the $r^{i}$ with $0\le i \le 2m$, as well as the space spanned by the $r^{2i}(r^{2}+r)$, $0\le i \le m-1$.
\end{theorem}

\begin{proof}
We may assume $m>0$. Consider the $Z$-submodule of $Z[[x]]$ consisting of those elements of $Z[[x]]$ that are expansions at infinity of weight $4m$ modular form of level $\Gamma_{0}(5)$. This module is stabilized by each classical Hecke operator $T_{p}:\sum c_{n}x^{n}\rightarrow \sum c_{pn}x^{n}+p^{4m-1}\sum c_{n}x^{pn}$, $p\ne 5$. Reducing mod~2 and using Theorem \ref{theorem1.11} we get the result.
\qed
\end{proof}

\begin{remark*}{Remark}
\label{remark1.15}
Multiplication by $G^{2}$ stabilizes $\modd$. Since $(F+G)^{6}=FG$, Theorem \ref{theorem1.13} shows that $\modd$, viewed as $Z/2[G^{2}]$-module, is spanned by $F$, $F^{3}$, $F^{5}$, $G$, $F^{2}G$ and $F^{4}G$. As $F$ has degree 6 over $Z/2(G)$, a basis of $\modd$ as $Z/2[G^{2}]$-module is $\{G, F, F^{2}G, F^{3}, F^{4}G, F^{5}\}$.
\end{remark*}

\begin{definition}
\label{def1.16}\hspace{2em}\\
\vspace{-4ex}
\begin{enumerate}
\item[(1)] $N2\subset \modd$ has $Z/2[G^{2}]$-basis $\{G, F, F^{2}G, F^{3}, F^{4}G\}$.
\item[(2)] $N1\subset N2$ has $Z/2[G^{2}]$-basis $\{G\}$.
\end{enumerate}
\end{definition}

\begin{definition}
\label{def1.17}
$J_{1}$, $J_{3}$, $J_{5}$, $J_{7}$ and $J_{9}$ are $F$, $F^{8}\!/G$, $G$, $F^{2}G$ and $F^{4}G$.
\end{definition}

$J_{1}$, $J_{5}$, $J_{7}$ and $J_{9}$ are evidently in $N2$. Since $(F+G)^{8}=FG(F+G)^{2}$, $F^{8}=G^{8}+FG(F+G)^{2}$, and $J_{3}=G^{7}+F(F+G)^{2}$ is also in $N2$. Now $G$, $F$, $F^{2}G$, $F^{3}$ and $F^{4}G$ evidently generate the same $Z/2[G^{2}]$-module as do $J_{5}$, $J_{1}$, $J_{7}$, $J_{3}$ and $J_{9}$. Consequently:

\begin{theorem}
\label{theorem1.18}\hspace{2em}\\
\vspace{-4ex}
\begin{enumerate}
\item[(1)] $J_{1}$, $J_{3}$, $J_{5}$, $J_{7}$ and $J_{9}$ are a $Z/2[G^{2}]$-basis of $N2$, while $J_{1}$, $J_{3}$, $J_{7}$, $J_{9}$ are a $Z/2[G^{2}]$-basis of $N2/N1$.
\item[(2)] Define $J_{k}$, $k$ odd, $k>0$, by taking $J_{k+10}$ to be $G^{2}J_{k}$. Then the $J_{k}$ with $(k,10)=1$ are a $Z/2$-basis of $N2/N1$.
\end{enumerate}
\end{theorem}

\begin{remark*}{Remarks}
\hspace{2em}\\
\vspace{-4ex}
\begin{enumerate}
\item[(1)] The space spanned by the $F^{k}$. $k$ odd and $>0$, is stabilized by the $T_{p}$ with $p\ne 2$. It follows that $N1$ is stabilized by the $T_{p}$ with $p\ne 2$ or $5$.
\item[(2)] One may describe $N2$ more elegantly as follows. $Z/2(F,G)$ is a degree 6 field extension of $Z/2(G)$, and we have a trace map, $Z/2(F,G)\rightarrow Z/2(G)$. Using the identity $(F+G)^{6}=FG$ we find that $F^{i}$, $0\le i\le 4$, have trace $0$, while $F^{5}$ has trace $G$. So $N2$ consists of those elements of $\modd$ of trace $0$. We'll see in the next section that the $T_{p}$, $p\ne 2$ or $5$ stabilize not only $\modd$ and $N1$, but also $N2$.
\end{enumerate}
\end{remark*}

\section{The spaces \bm{$W$}, \bm{$W_{a}$} and \bm{$W_{b}$}. A decomposition of \bm{$N2/N1$}}
\label{section2}

\begin{definition}
\label{def2.1}
$H=F(x^{25})=G(x^{5})$.
\end{definition}

As we noted in section \ref{section1}, $D=\sum_{(n,10)=1,\ n>0}x^{n^{2}}$ is $F+H$.

\begin{definition}
\label{def2.2}
$pr: Z/2[[x]]\rightarrow Z/2[[x]]$ takes $\sum c_{n}  x^{n}$ to $\sum_{(n,5)=1}  c_{n} x^{n}$.
\end{definition}

Since $G$ lies in $Z/2[[x^{5}]]$, $pr$ is $Z/2[G]$-linear and $pr(N1)=0$. The effect of $pr$ on the $Z/2[G^{2}]$-basis $\{J_{1}, J_{3}, J_{7}, J_{9}\}$ of $N2/N1$ is easily described: 

\begin{lemma}
\label{lemma2.3}
$pr$ takes $J_{1}, J_{3}, J_{7}, J_{9}$ to $D, D^{8}\!/G, D^{2}G, D^{4}G$.
\end{lemma}

\begin{proof}
$pr(J_{3}) = pr(F^{8}\!/G) = pr((H^{8}+D^{8})\!/G) = pr(D^{8}\!/G)$. Since all exponents appearing in $D$ are prime to 5, the same holds for all exponents appearing in $D^{8}\!/G$ and $pr(D^{8}\!/G)=D^{8}\!/G$. The other results have similar proofs.
\qed
\end{proof}

\begin{lemma}
\label{lemma2.4}
$D^{15}+G^{4}D^{3}+G^{3}=0$.
\end{lemma}

\begin{proof}
$(F+G)^{6}= FG$. Replacing $x$ by $x^{5}$ we find that $(G+H)^{6}= GH$. So $(D+F+G)^{6}= G(D+F)$. Adding $FG$ to both sides, expanding in powers of $D$ and dividing by $D$ we find that $D^{5}+(F+G)^{2}D^{3}+ (F+G)^{4}D=G$. So if we set $A=(F+G)^{2}$, then:

\begin{proofresult}
$DA^{2}+D^{3}A + (D^{5}+G)=0$.
\end{proofresult}

On the other hand, $A^{3}=(F+G)^{6}=FG$. So $A^{6}=(A+G^{2})G^{2}$, and:

\begin{proofresult}
$A^{6}+G^{2}A + G^{4}=0$.
\end{proofresult}

We can now eliminate $A$ from (1) and (2) to get our result. Explicitly, when we multiply (1) by $DA^{4}+D^{3}A^{3}+GA^{2}+D^{7}A$, (2) by $D^{2}$, and add, all terms involving $A^{6}$, $A^{5}$, $A^{4}$ or $A^{3}$ drop out, and we find that $(D^{10}+D^{5}G+G^{2})(A^{2}+D^{2}A) + D^{2}G^{4}=0$. Let $B=A^{2}+D^{2}A$. Then $(D^{10}+D^{5}G+G^{2})B=D^{2}G^{4}$. Also, (1) tells us that $DB=D^{5}+G$. So $(D^{10}+D^{5}G+G^{2})(D^{5}+G)=D^{3}G^{4}$, as desired.
\qed
\end{proof}

Multiplying by $D^{5}\!/G^{4}$ gives:

\begin{corollary}
\label{corollary2.5}
$(D^{5}\!/G)^{4}+D^{5}\!/G=D^{8}$.
\end{corollary}

\begin{definition}
\label{def2.6}
When $(i,10)=1$, $D_{i}=pr(J_{i})$.
\end{definition}

As we've seen, $D_{1}$, $D_{3}$, $D_{7}$ and $D_{9}$ are $D$, $D^{8}\!/G$, $D^{2}G$ and $D^{4}G$. And since $pr$ is $Z/2[G^{2}]$-linear, $D_{i+10}=G^{2}D_{i}$.

\begin{theorem}
\label{theorem2.7}
The $D_{i}$ are linearly independent over $Z/2$. So $pr$ maps $N2/N1$ bijectively to the space $W$ spanned by the $D_{i}$. (Recall that the $J_{k}$ with $(k,10)=1$ are a $Z/2$-basis of $N2/N1$.)
\end{theorem}

\begin{proof}
$G$ is transcendental over $Z/2$. It follows from Lemma \ref{lemma2.4} that $D$ has degree 15 over $Z/2(G)$, so that $D$, $D^{8}$, $D^{2}$ and $D^{4}$ are linearly independent over this field. Consequently, $D_{1}=D$, $D_{3}=D^{8}\!/G$, $D_{7}=D^{2}G$ and $D_{9}=D^{4}G$ are linearly independent over $Z/2[G^{2}]$, giving the result.
\qed
\end{proof}

\begin{remark*}{Remark}
$W$ does not consist of mod~2 modular forms of level~5. In fact the elements of $W$ ``are of level 25.''
\end{remark*}

We shall see that the $T_{p}$ ($p\ne 2$ or $5$ as usual) stabilize both $W$ and $N2$. To this end we use a real Dirichlet character $\chi$, of modulus 20, to write $W$ as a direct sum of $Z/2[G^{4}]$-submodules, $W_{a}$ and $W_{b}$. Then we show that the $T_{p}$ with $\chi(p)=1$ stabilize $W_{a}$ and $W_{b}$, while those with  $\chi(p)=-1$ take $W_{a}$ to $W_{b}$ and $W_{b}$ to $W_{a}$.

\begin{definition}
\label{def2.8}
$\chi$ is the mod~20 Dirichlet character taking $1,3,7,9$ to $1$, and taking $11,13,17,19$ to $-1$.
\end{definition}

\begin{definition}
\label{def2.9}
$W_{a}\ (\mbox{resp.}\ W_{b})$ is the subspace of $W$ spanned by the $D_{i}$ with $\chi(i)=1\ (\mbox{resp.}\ -1)$.
\end{definition}

Evidently $W_{a}$ and $W_{b}$ are $Z/2[G^{4}]$-submodules of $W$ with bases $\{D_{1},D_{3},D_{7},D_{9}\}$ and $\{D_{11},D_{13},D_{17},D_{19}\}$. And $W=W_{a}\oplus W_{b}$.

\begin{lemma}
\label{lemma2.10}
If $x^{n}$ appears in $D_{k}$, $n$ is $k$ or $9k\bmod{40}$.  In particular when $x^{n}$ appears in $D_{k}$, $\chi(n)=\chi(k)$.
\end{lemma}

\begin{proof}
Since $D_{k+10}=G^{2}D_{k}$ and all exponents appearing in $G^{2}$ are $10\bmod{40}$, it suffices to prove the lemma for $k$ in$\{1,3,7,9\}$. Suppose for example $k=9$, so that $D_{k}=D^{4}G$. If $(i,10)=1$, $i^{2}\equiv 1$ or $9\bmod{40}$. So all exponents appearing in $D$ are $1$ or $9\bmod{40}$, and all exponents in $D^{4}G$ are $4+5=9$ or $36+5=41\bmod{40}$. The other cases are handled similarly.
\qed
\end{proof}

\begin{definition}
\label{def2.11}
\hspace{2em}\\
\vspace{-4ex}
\begin{enumerate}
\item[] $p_{a}:Z/2[[x]]\rightarrow Z/2[[x]]$ takes $\sum c_{n}x^{n}$ to $\sum_{\chi(n)=1} c_{n}x^{n}$.
\item[] $p_{b}:Z/2[[x]]\rightarrow Z/2[[x]]$ takes $\sum c_{n}x^{n}$ to $\sum_{\chi(n)=-1} c_{n}x^{n}$.
\end{enumerate}
\end{definition}

$p_{a}$ and $p_{b}$ are evidently $Z/2[G^{4}]$-linear. Furthermore:

\begin{lemma}
\label{lemma2.12}
$p_{a}(G^{2}f)$ and $p_{b}(G^{2}f)$ are $G^{2}p_{b}(f)$ and $G^{2}p_{a}(f)$.
\end{lemma}

\begin{proof}
We may assume that all the exponents appearing in $f$ are congruent to some fixed $k\bmod{20}$. Then the exponents in $G^{2}f$ are congruent to $k+10$, and we use the fact that $\chi(k+10)=-\chi(k)$.
\qed
\end{proof}

\begin{lemma}
\label{lemma2.13}
\hspace{2em}\\
\vspace{-4ex}
\begin{enumerate}
\item[] If $\chi(k)=1$, $p_{a}(J_{k})=D_{k}$ and $p_{b}(J_{k})=0$.
\item[] If $\chi(k)=-1$, $p_{a}(J_{k})=0$ and $p_{b}(J_{k})=D_{k}$.
\end{enumerate}
\end{lemma}

\begin{proof}
Since $pr(J_{k})=D_{k}$, $p_{a}(J_{k})=p_{a}(D_{k})$ while $p_{b}(J_{k})=p_{b}(D_{k})$.  Lemma \ref{lemma2.10} then gives the result.
\qed
\end{proof}

As $Z/2[G^{4}]$-module, $N2/N1$ has basis $\{J_{1},J_{3},J_{7},J_{9},J_{11},J_{13},J_{17},J_{19}\}$. Evidently $N2/N1=N2a\oplus N2b$, where $N2a$ and $N2b$ are the $Z/2[G^{4}]$-submodules with bases $\{J_{1},J_{3},J_{7},J_{9}\}$ and $\{J_{11},J_{13},J_{17},J_{19}\}$. The $J_{k}$ with $\chi(k)=1$ are a $Z/2$-basis of $N2a$; those with $\chi(k)=-1$ are a $Z/2$-basis of $N2b$. Lemma \ref{lemma2.13} now shows that $p_{a}$ maps $N2/N1$ onto $W_{a}$ with kernel $N2b$, while $p_{b}$ maps $N2/N1$ onto $W_{b}$ with kernel $N2a$.

Suppose now that $f$ is in $\modd/N1$. We'll show that if $p_{b}(f)=0$, then $f$ is in $N2a$, while if $p_{a}(f)=0$, $f$ is in $N2b$.  So in either case, $f$ is in $N2/N1$. This is key to showing that the $T_{p}$ stabilize $N2$.

\begin{lemma}
\label{lemma2.14}
$p_{a}(F(F+G)^{4})=p_{a}(r^{8}G)$ and $p_{b}(F(F+G)^{4})=p_{b}(r^{8}G)$.
\end{lemma}

\begin{proof}
$F(F+G)^{4}=r(r+1)^{5}(r^{2}+r)^{4}=(r+1)^{8}(r^{6}+r^{5})=r^{8}G+G$. Now apply $p_{a}$ and $p_{b}$.
\qed
\end{proof}

\begin{lemma}
\label{lemma2.15}
Let $S=p_{a}(r^{8}G)$, $T=p_{b}(r^{8}G)$. Then $T=D^{5}+G$, $S=D^{10}\!/G + G$.
\end{lemma}

\begin{proof}
Since $r+r^{2}=F+G$, $r^{8}G=G((F+G)^{8}+(F+G)^{16}+(F+G)^{32}+\cdots )$. So $pr(r^{8}G)=G(D^{8}+D^{16}+D^{32}+\cdots)$. Now all exponents $n$ appearing in any of $GD^{8}$, $GD^{32}$, $GD^{128}$, \ldots are $13$ or $37\bmod{40}$, while those in any of $GD^{16}$, $GD^{64}$, \ldots are $21$ or $69\bmod{40}$. Applying $p_{b}$ to our identity we find that $T=G(D^{8}+D^{32}+D^{128}+\cdots)$. Then $(T/G)^{4}+(T/G)=D^{8}$. So by Corollary \ref{corollary2.5}, $T/G$ and $D^{5}\!/G$ differ by a constant, and comparing expansions we see that the constant is $1$. So $T=D^{5}+G$. Finally $S=G(D^{16}+D^{64}+D^{256}+\cdots )=T^{2}\!/G$.
\qed
\end{proof}

\begin{lemma}
\label{lemma2.16}
\hspace{2em}\\
\vspace{-4ex}
\begin{enumerate}
\item[(1)] \parbox{2.5in}{$p_{a}(F(F+G)^{4})=D^{10}\!/G + G$,} $p_{b}(F(F+G)^{4})=D^{5}+G$. 
\item[(2)] \parbox{2.5in}{$p_{a}(FG^{2}(F+G)^{4})=D^{5}G^{2}+G^{3}$,} $p_{b}(FG^{2}(F+G)^{4})=D^{10}G+G^{3}$.
\end{enumerate}
\end{lemma}

\begin{proof}
Lemmas \ref{lemma2.14} and \ref{lemma2.15} give (1), and Lemma \ref{lemma2.12} yields (2).
\qed
\end{proof}

We saw in the remark following the proof of Theorem \ref{theorem1.14} that a basis of $\modd$ as $Z/2[G^{2}]$-module is  $\{G, F, F^{2}G, F^{3}, F^{4}G, F^{5}\}$. The last five of those elements then give a  basis of $\modd/N1$. It follows that another $Z/2[G^{2}]$-basis of
$\modd/N1$ is $\{J_{1},J_{3},J_{7},J_{9},F(F+G)^{4}\}$. Then a $Z/2[G^{4}]$-basis of $\modd/N1$ is  $\{J_{1},J_{3},J_{7},J_{9},J_{11},J_{13},J_{17},J_{19},F(F+G)^{4},FG^{2}(F+G)^{4}\}$.


\begin{theorem}
\label{theorem2.17}
The kernels of $p_{b}:\modd/N1\rightarrow Z/2[[x]]$ and $p_{a}:\modd/N1\rightarrow Z/2[[x]]$ are $N2a$ and $N2b$ where $N2a$ and $N2b$ have $Z/2[G^{4}]$-module bases  $\{J_{1},J_{3},J_{7},J_{9}\}$ and $\{J_{11},J_{13},J_{17},J_{19}\}$.
\end{theorem}

\begin{proof}
Consider the 10 element $Z/2[G^{4}]$-module basis of $\modd/N1$ given in the sentence preceding Theorem \ref{theorem2.17}. $p_{b}$ annihilates $J_{1}$, $J_{3}$, $J_{7}$ and $J_{9}$ and sends the last 6 basis elements to $D_{11}=DG^{2}$, $D_{13}=D^{8}G$, $D_{17}=D^{2}G^{3}$, $D_{19}=D^{4}G^{3}$, $D^{5}+G$ and $D^{10}G+G^{3}$; see Lemma \ref{lemma2.16}. As we've seen, $D$ has degree 15 over $Z/2(G)$, and so $D$, $D^{8}$, $D^{2}$, $D^{4}$, $D^{5}+G$ and $D^{10}+G^{2}$ are linearly independent over this field. So no non-trivial $Z/2[G^{4}]$-linear combination of $D_{11}$, $D_{13}$, $D_{17}$, $D_{19}$, $D^{5}+G$ and $D^{10}+G^{3}$ is zero, and the result for $p_{b}$ follows. The result for $p_{a}$ is proved similarly.
\qed
\end{proof}

\begin{corollary}
\label{corollary2.18}
\hspace{2em}\\
\vspace{-4ex}
\begin{enumerate}
\item[(1)] If $\chi(p)=1$, $T_{p}$ stabilizes $N2a$ and $N2b$.
\item[(2)] If $\chi(p)=-1$, $T_{p}$ maps $N2a$ to $N2b$ and $N2b$ to $N2a$.
\end{enumerate}
\end{corollary}

\begin{proof}
For example suppose $f$ is in $N2b$ with $\chi(p)=-1$. Then the exponents, $n$, appearing in (a representative of) $f$ have $\chi(n)=-1$ or $0$. Since $\chi$ is multiplicative, the exponents appearing in $T_{p}$ applied to this representative have $\chi(n)=1$ or $0$. In other words, $T_{p}(f)$ is an element of $\modd/N1$ in the kernel of $p_{b}$. By Theorem \ref{theorem2.17}, $T_{p}(f)$ is in $N2a$. The other cases are treated similarly.
\qed
\end{proof}

\begin{theorem}
\label{theorem2.19}
If $p\ne 2$ or $5$, $T_{p}$ stabilizes $N2$ and $W$.
\end{theorem}

\begin{proof}
If $h$ is in $N2$, the image of $h$ in $N2/N1$ is the sum of an element of $N2a$ and an element of $N2b$. By Corollary \ref{corollary2.18}, the same holds for the image of $T_{p}(h)$ in $\modd/N1$. So this image is in $N2/N1$ and $h$ is in $N2$. Suppose $f$ is in $W$. Then $f=pr(h)$ with $h$ in $N2$. Then $T_{p}(h)$ is in $N2$, and $T_{p}(f)=pr(T_{p}(h))$ is in $W$.
\qed
\end{proof}

\begin{corollary}
\label{corollary2.20}
Suppose $p\ne 2$ or $5$, and $(n,10)=1$. Then $T_{p}(D_{n})$ is a sum of distinct $D_{k}$. In such a decomposition each $k$ is either $pn$ or $9pn\bmod{40}$, and in particular $\chi(k)=\chi(p)\chi(n)$.
\end{corollary}

\begin{proof}
Since $D_{n}$ is in $W$, so is $T_{p}(D_{n})$, giving the first result. Also the exponents appearing in $T_{p}(D_{n})$ are all congruent to $pn$ or $9pn\bmod{40}$ (by Lemma \ref{lemma2.10} and the definition of $T_{p}$) while those appearing in a $D_{k}$ are $k$ or $9k\bmod{40}$. It follows easily that those $D_{k}$ appearing in the sum for which $k$ is neither congruent to $pn$ nor to $9pn\bmod{40}$ sum to $0$. So there are no such $k$.
\qed
\end{proof}

\begin{remark*}{Remark}
Since $N2/N1=N2a\oplus N2b$, $pr\hspace{-.4em}:\hspace{-.3em}N2/N1 \rightarrow Z/2[[x]]$ is 1--1 with image $W_{a}\oplus W_{b}=W$. So $pr$ gives an identification of the subquotient $N2/N1$ of $\modd$ with $W$. Now the $T_{p}$, $p\ne 2$ or $5$, stabilize $N2/N1$ and $W$, and the identification preserves the action of $T_{p}$. Corollary \ref{corollary2.20} and this identification are the quoted results (1) and (2) at the beginning of section \ref{section1}.
\end{remark*}

We show next that when $p\equiv 3$ or $7\mod{10}$, each $k$ in Corollary \ref{corollary2.20} is $<n$. (This is also true when $p\equiv 1$ or $9$, but this will only be proved in the final section.)

\begin{definition}
\label{def2.21}
$v_{2}=r^{2}+r$, $v_{4}=r^{4}+r^{3}$, $v_{6}=r^{6}+r^{5}$, $v_{10}=r^{10}+r^{9}+r^{8}+r^{7}$, $v_{12}=r^{12}+r^{11}+r^{10}+r^{9}$.
\end{definition}

Note that $v_{2}=F+G$, $v_{4}=(F+G)^{3}+G$, $v_{6}=G$, $v_{10}=(F+G)^{2}G$ and $v_{12}=(F+G)^{4}G+(F+G)G^{2}$. So $v_{2}$, $v_{4}$, $v_{6}$, $v_{10}$ and $v_{12}$ generate the same $Z/2[G^{2}]$-submodule of $\modd$ as do $G$, $F$, $F^{2}G$, $F^{3}$ and $F^{4}G$; that is to say they generate $N2$. Since $G^{2}=r^{12}+r^{10}$, $G^{2s}v_{j}$ is an element of $N2$ whose degree in $r$ is $12s+j$.

\begin{lemma}
\label{lemma2.22}
$T_{p}(D_{10m+3})$ and $T_{p}(D_{10m+7})$ are sums of $D_{k}$ with $k\le 10m+7$. 
\end{lemma}

\begin{proof}
$J_{7}=F^{2}G\equiv v_{10}\bmod{N1}$. And $J_{3}=F^{8}\!/G=(F+G)^{8}\!/G+G^{7}=F(F+G)^{2}+G^{7}$. So mod $N1$, $J_{3}+J_{7}\equiv (F+G)^{3}\equiv v_{4}$. It follows that $J_{10m+3}=G^{2m}J_{3}$ and $J_{10m+7}=G^{2m}J_{7}$ are each congruent mod $N1$ to polynomials in $r$ of degree $\le 12m+10$. By Theorem \ref{theorem1.14} the same is true of $T_{p}(J_{10m+3})$ and $T_{p}(J_{10m+7})$. (Note also that $J_{1}=F\equiv v_{2}\bmod{N1}$ while $J_{9}+J_{11}=F^{4}G+FG^{2}\equiv v_{12}$.) Now take an $h$ of degree $\le 12m+10$ in $r$ which is congruent mod $N1$ to $T_{p}(J_{10m+3})$ (or to $T_{p}(J_{10m+7})$). Since $J_{10m+3}$ and $J_{10m+7}$ are in $N2$, so is $h$. Write $h$ as a sum of distinct $G^{2s}v_{i}$ with each $i$ in $\{2,4,6,10,12\}$. The degree in $r$ of $G^{2s}v_{i}$ is $12s+i$. These degrees are distinct, and since the degree of $h$ is $\le 12m+10$, the $s$ appearing in those $G^{2s}v_{i}$ with $i=12$ are all $<m$, while the remaining $s$ are all $\le m$. Since $v_{2}$, $v_{4}$, $v_{6}$, $v_{10}$ and $v_{12}$ are congruent mod $N1$ to $J_{1}$, $J_{3} +J_{7}$, $0$, $J_{7}$ and $J_{9}+J_{11}$, and each of $10m+1$, $10m+7$, $10m+7$, $10(m-1)+11$ is $\le 10m+7$, each $G^{2s}v_{i}$ appearing in the sum for $h$ is, mod $N1$, a sum of $J_{k}$ with $k\le 10m+7$. Applying $pr$ to the identity $h=\sum G^{2s}v_{i}$, noting that $pr$ preserves the action of $T_{p}$ and that $pr(J_{k})$ is either $D_{k}$ or $0$, we get the result.
\qed
\end{proof}

\begin{lemma}
\label{lemma2.23}
$T_{p}(D_{10m+9})$ and $T_{p}(D_{10m+11})$ are sums of $D_{k}$ with $k\le 10m+11$. 
\end{lemma}

\begin{proof}
$J_{11}=G^{2}J_{1}\equiv G^{2}v_{2}\bmod{N1}$, while $J_{9}+J_{11}\equiv v_{12}$. It follows that $J_{10m+9}$ and $J_{10m+11}$ are congruent mod $N1$ to polynomials in $r$ of degree $\le 12m+14$. By Theorem \ref{theorem1.14}, the same is true of $T_{p}(J_{10m+9})$ and $T_{p}(J_{10m+11})$. Now take an $h$ of degree $\le 12m+14$ in $r$ which mod $N1$ is $T_{p}(J_{10m+9})$ (or $T_{p}(J_{10m+11})$). $h$ is in $N2$ and we write it as a sum of distinct $G^{2s}v_{i}$, $i$ in $\{2,4,6,10,12\}$. Arguing as in the proof of Lemma \ref{lemma2.22} we find that the $s$ appearing in the $G^{2s}v_{i}$ with $i=2$ are $\le m+1$, while the remaining $s$ are $\le m$, and we continue as in the proof of Lemma \ref{lemma2.22}.
\qed
\end{proof}

\begin{theorem}
\label{theorem2.24}
Suppose $p\equiv 3$ or $7\mod{10}$. When we write $T_{p}(D_{n})$ as a sum of distinct $D_{k}$, each $k<n$.
\end{theorem}

\begin{proof}
Suppose $n=10m+3$ or $10m+7$. By Lemma \ref{lemma2.22}, $T_{p}(D_{n})$ is a sum of distinct $D_{k}$, $k\le 10m+7$. By Corollary \ref{corollary2.20}, each $k\equiv pn$ or $9pn\mod{10}$ and so is $1$ or $9$ mod~10. So no $k$ can be $10m+3$ or $10m+7$. If  $n=10m+9$ or $10m+11$, $T_{p}(D_{n})$ is, by Lemma \ref{lemma2.23}, a sum of $D_{k}$, $k\le 10m+11$. Then each $k\equiv pn$ or $9pn\mod{10}$ and so is $3$ or $7$ mod~10. So no $k$ is $10m+9$ or $10m+11$. Finally, $T_{p}(D_{1})=T_{p}(D)=0$.
\qed
\end{proof}

We shall write down a linear recursion satisfied by the $T_{3}(D_{n})$, $n\equiv 1,3,7,9\mod{20}$. This recursion, together with some initial condition results, proved with the help of Theorem \ref{theorem2.24} when $p=3$, allow us to relate $T_{3}(D_{n})$ to a polynomial $P_{n}$ appearing in section 5 of \cite{2}.

\begin{lemma}
\label{lemma2.25}
For $u$ in $Z/2[[x]]$, $T_{3}(G^{16}u)=G^{16}T_{3}(u)+G^{4}T_{3}(G^{4}u)$.
\end{lemma}

\begin{proof}
Let $u$ be the $2$-variable polynomial $A^{4}+B^{4}+AB$. We have the ``level 3 modular equation for $F$,'' $U(F(x^{3}),F(x))=0$. Replacing $x$ by $x^{5}$ shows that $U(G(x^{3}),G(x))=0$. Now proceed as in the proof of Lemma 2.19 of \cite{1} (though now we have 4 imbeddings $\varphi_{i}$, the first taking $f(x)$ to $f(x^{3})$, and the others taking $f(x)$ to $f(\lambda x^{1/3})$ where $\lambda$ runs over the cube roots of $1$ in an algebraic closure of $Z/2$).
\qed
\end{proof}

\begin{theorem}
\label{theorem2.26}
The $T_{3}(D_{n})$. $n\equiv 1,3,7,9\mod{20}$ satisfy the recursion $T_{3}(D_{n+80})=G^{16}T_{3}(D_{n})+G^{4}T_{3}(D_{n+20})$. And $T_{3}$ takes 
\vspace{-1ex}
\[
D_{1},D_{3},D_{7},D_{9},D_{21},D_{23},D_{27},D_{29},D_{41},D_{43},D_{47},D_{49},D_{61},D_{63},D_{67},D_{69}
\]
\vspace{-4ex}
to:
\vspace{-4ex}
\begin{eqnarray*}
0,D_{1},0,D_{3},D_{7},D_{21},D_{9},D_{23},0,D_{41},D_{21},D_{43}+D_{27},\\
D_{47}+D_{23},D_{61}+D_{29}+D_{21},D_{49}+D_{41},D_{63}+D_{47}+D_{23}.
\end{eqnarray*}
\vspace{-4ex}
\end{theorem}

\begin{proofsketch}
Taking $u=D_{n}$ in Lemma \ref{lemma2.25}, we ge the recursion. I'll calculate $T_{3}(D_{47})$ explicitly, the other 15 initial values being derived in a similar way. By Corollary \ref{corollary2.20} and Theorem \ref{theorem2.24}, $T_{3}(D_{47})$ is a $Z/2$-linear combination of $D_{21}$ and $D_{29}$. Now $D_{47}=G^{8}(D^{2}G)=(x^{40}+x^{360}+\cdots)(x^{2}+x^{18}+x^{98}+\cdots)(x^{5}+x^{45}+\cdots)$. So the coefficients of $x^{7}$, $x^{63}$ and $x^{87}$ in $D_{47}$ are $0$, $1$ and $1$. It follows that the coefficients of $x^{21}$ and $x^{29}$ in $T_{3}(D_{47})$ are each $1$. But $D_{21}=G^{4}D=x^{21}+x^{29}+\cdots$ while $D_{29}=x^{29}+\cdots$. It follows that $T_{3}(D_{47})=D_{21}$. 
\qed
\end{proofsketch}

Now let $w$ be an indeterminate over $Z/2$.

\begin{definition}
\label{def2.27}
$V^{\prime}\subset Z/2[w]$ is the space spanned by the $w^{k}$ with $k\equiv 1,3,7,9\mod{20}$.
\end{definition}

There is a $Z/2$-linear identification of $W_{a}$ with $V^{\prime}$ taking $D_{k}$ to $w^{k}$. Then $T_{3}:W_{a}\rightarrow W_{a}$ goes over to a map $V^{\prime}\rightarrow V^{\prime}$ which we'll still call $T_{3}$. Now in section 5 of \cite{2} (see Theorem 5.1, the paragraphs preceding it, and Remark 5.2) we defined certain $P_{k}$, $k\equiv 1,3,7,9\mod{20}$ in $V^{\prime}$.

\begin{theorem}
\label{theorem2.28}
$T_{3}:V^{\prime}\rightarrow V^{\prime}$ takes $w^{k}$ to the $P_{k}$ of \cite{2} section 5.
\end{theorem}

\begin{proof}
Let $A_{k}=T_{3}(w^{k})$. The recursion of Theorem \ref{theorem2.26} tells us that $A_{k+80}=w^{80}A_{k}+w^{20}A_{k+20}$, while Theorem 5.1 of \cite{2} tells us that $P_{k+80}=w^{80}P_{k}+w^{20}P_{k+20}$. Putting the initial values in Theorem \ref{theorem2.26} together with Theorem 5.1 of \cite{2} we see that $P_{k}=A_{k}$ whenever $k<80$, $k\equiv 1,3,7,9\mod{20}$. The recursions then show that $P_{k}=A_{k}$ whenever $k\equiv 1,3,7,9\mod{20}$.
\qed
\end{proof}

\section{Coding a basis of \bm{$W_{a}$}. The effect of \bm{$T_{3}$} on the code}
\label{section3}

Define a total ordering on $N\times N$ as follows.  $(c,d)$ ``precedes'' (or ``is earlier than'') $(a,b)$ if $c+d<a+b$ or $c+d=a+b$ and $d<b$.

Consider the space $W_{a}$ with its basis $\{D_{k}, k\equiv 1,3,7,9\mod{20}\}$. In this section we ``code the basis'' by identifying it with $N\times N$; $(c,d)^{*}$ will denote the particular $D_{k}$ corresponding to $(c,d)$. Under this identification, the total ordering on $N\times N$ given in the last paragraph goes over to a total ordering on our basis, and we use the ``precedence'' language of the last paragraph and show:

\begin{enumerate}
\item[(1)] $T_{3}(a,b)^{*}$ is a sum of $(c,d)^{*}$ with $c+d<a+b$.
\item[(2)] If $a>0$, $T_{3}(a,b)^{*}=(a-1,b)^{*}+$ a sum of earlier $D_{j}$.
\end{enumerate}

The hard work in establishing (1) and (2) has already been done in Lemma 5.5 of \cite{2} and Theorem \ref{theorem2.28} of the last section, so most of this section is a summary of results from \cite{2}.

\begin{definition}
\label{def3.1}
$g:N\rightarrow N$ is the function with $g(2n)=4g(n)$ and $g(2n+1)=g(2n)+1$.
\end{definition}

\begin{definition}
\label{def3.2}
$V\subset Z/2[t]$ is spanned by the $t^{k}$, $k$ odd. If $(a,b)$ is in $N\times N$, $[a,b]$ in $V$ is $t^{1+2g(a)+4g(b)}$.
\end{definition}

$(a,b)\rightarrow [a,b]$ sets up a 1--1 correspondence between $N\times N$ and the monomial basis of $V$. The total ordering of $N\times N$ goes over to a total ordering of the basis. Once again we use the language of ``precedence.''  We now pass from $V$ to $V^{\prime}$. 

\begin{definition}
\label{def3.3}
If $V^{\prime}$ is as in Definition \ref{def2.27}, $\varphi:V\rightarrow V^{\prime}$ is the $Z/2$-linear bijection with:

\vspace{-3ex}
\begin{eqnarray*}
&&\varphi(t^{16m+1}) = w^{40m+1},\quad \varphi(t^{16m+3}) = w^{40m+3}\\
&&\varphi(t^{16m+5}) = w^{40m+7},\quad \varphi(t^{16m+5}) = w^{40m+21}\\
&&\varphi(t^{16m+9}) = w^{40m+9},\quad \varphi(t^{16m+11}) = w^{40m+27}\\
&&\varphi(t^{16m+13}) = w^{40m+23},\quad \varphi(t^{16m+15}) = w^{40m+29}
\end{eqnarray*}
\end{definition}

\begin{definition}
\label{def3.4}
$\langle a,b\rangle$ in $V^{\prime}$ is $\varphi[a,b]$. $(a,b)^{*}$ in $W_{a}$ is the image of $\langle a,b\rangle$ under the identification of $V^{\prime}$ with $W_{a}$ taking $w^{k}$ to $D_{k}$.
\end{definition}

Note that $\varphi:V\rightarrow V^{\prime}$ identifies the monomial basis (in $t$) of $V$ with the monomial basis (in $w$) of $V^{\prime}$. So Definition \ref{def3.4} results in a coding of the basis $\{D_{k}\}$ of $W_{a}$.

\begin{theorem}
\label{theorem3.5}
\hspace{2em}\\
\vspace{-4ex}
\begin{enumerate}
\item[(1)] $T_{3}(a,b)^{*}$ is a sum of $(c,d)^{*}$ with $c+d<a+b$.
\item[(2)] If $a>0$, $T_{3}(a,b)^{*}=(a-1,b)^{*}+$ a sum of earlier $D_{j}$.
\end{enumerate}
\end{theorem}

\begin{proof}
We saw in the last section that $T_{3}$ stabilizes $W_{a}$. In view of the identification of $W_{a}$ with $V^{\prime}$ it suffices to show that $T_{3}\langle a,b \rangle$ is a sum of $\langle c,d\rangle$ with $c+d<a+b$, and that if $a>0$, $T_{3}\langle a,b \rangle =\langle a-1,b \rangle +$ a sum of earlier monomials. Suppose then that $\langle a,b \rangle =w^{k}$. By Theorem \ref{theorem2.28}, $T_{3}\langle a,b \rangle$ is the $P_{k}$ of \cite{2}, Theorem 5.1.  So our result is precisely Lemma 5.5 of \cite{2}.
\qed
\end{proof}

We will need one further property of our code.

\begin{theorem}
\label{theorem3.6}
If $D_{i}=(c,d)^{*}$ precedes $D_{j}=(0,b)^{*}$, then $i<j$.
\end{theorem}

\begin{proof}
In view of our identification of $W_{a}$ with $V^{\prime}$, it's enough to show that if $w^{i}=\langle c,d \rangle$ precedes $w^{j}=\langle 0,b\rangle$, then $i<j$.

Now there is the following corresponding result for $V$. If $t^{i}=[c,d]$ precedes $t^{j}=[0,b]$ then $i<j$. For $c+d\le b$, so by Lemma 4.1 of \cite{1}, $4g(c)+4g(d)\le 4g(b)$. Then $i=1+2g(c)+4g(d)\le 1+4g(b)=j$, and since $i\ne j$, $i<j$.

We'll deduce Theorem 3.6 from this result for $V$. Suppose first that $b$ is even, and let $m=g(b/2)$. Then $1+4g(b)=16m+1$, and $[0,b]=t^{16m+1}$. Write $[c,d]$ as $t^{16m^{\prime}+r}$ with $r$ in $\{1,3,5,7,9,11,13,15\}$. Since $\langle c,d \rangle$ precedes $\langle 0,b \rangle$, $[c,d]$ precedes $[0,b]$. By the result of the last paragraph, $16m^{\prime}+r < 16m+1$ and so $m^{\prime}<m$. Now $\langle 0,b \rangle = w^{40m+1}$ so that $j=40m+1$. Similarly, $i\le 40m^{\prime}+29\le 40m-11$, and this is $<j$. The argument when $b$ is odd is similar. Let $m=g((b-1)/2)$, so that $g(b)=4m+1$, and $[0,b]=t^{16m+5}$. Let $m^{\prime}$ and $r$ be as in the case of even $b$. Arguing as in the case of even $b$ we find that $16m^{\prime}+r<16m+5$. If $m^{\prime}<m$ we proceed as in the case of even $b$. If $m^{\prime}=m$ then $r$ must be $1$ or $3$. So $\langle c,d \rangle$ is either $w^{40m+1}$ or $w^{40m+3}$, while $\langle 0,b \rangle=\varphi(t^{16m+5})=w^{40m+7}$, and once again $i<j$.
\qed
\end{proof}

\section{Type \bm{$a$} ideals of \bm{$Z[\sqrt{-10}]$} and Gauss-classes}
\label{section4}

This section is the counterpart to section 3 of \cite{1}. Fix a power, $q$, of 2. We shall (essentially) use binary quadratic forms of discriminant $-640q^{2}$ and their associated theta-series to construct a subspace $DI(q)$ of $W_{a}$ of dimension $2q$, stable under the $T_{p}$ with $\chi(p)=1$, and annihilated by $T_{3}$. and we'll give a simple description of the action of $T_{7}$ on $DI(q)$ involving Gaussian composition of forms. We will not use all primitive positive forms of discriminant $-640q^{2}$; the Dirichlet character $\chi$ of Definition \ref{def2.8} may be thought of as a genus character, and we'll only consider forms on which this character takes the value $1$. We'll see that the $SL_{2}(Z)$-classes of such forms make up a cyclic group of order $4q$, and that the class of a form representing $7$ is a generator.

As in section 3 of \cite{1}, we'll avoid the explicit language of binary forms. Instead we'll consider ideals $I$ in $Z[\sqrt{-10}]$ for which $\chi(\mathrm{norm}(I))= 1$. We'll say that such an ideal is ``of type $a$.''  We fix a power $q$ of $2$, and introduce an equivalence relation, depending on $q$, on the type $a$ ideals. We will call our equivalence relation ``Gauss-equivalence,'' and the equivalence classes under it ``Gauss-classes.''

The class number of $Z[\sqrt{-10}]$ is $2$, with an ideal of norm $7$ representing the non-principal class. We begin by defining Gauss-equivalence for principal ideals $(b+c\sqrt{-10})$ of type $a$.

Since the norm of $(b+c\sqrt{-10})$ is $b^{2}+10c^{2}$, $(b+c\sqrt{-10})$ is of type $a$ precisely when $(b,10)=1$ and $c$ is even. Note also that the generator of a principal ideal is defined up to multiplication by $\pm 1$.

\begin{definition}
\label{def4.1}
Principal type $a$ ideals $(\alpha)$ and $(\beta)$ are equivalent if there is an integer $N$ with $(N,10)=1$ such that $N\alpha\equiv\beta\bmod{4q}$ in the ring $Z[\sqrt{-10}]$.
\end{definition}

Evidently this does not depend on the choices of generators for the ideals, and is an equivalence relation. Also ideal multiplication makes the set of equivalence classes into a semigroup. Since $(\alpha)(\bar{\alpha})=(\mathrm{norm}(\alpha))$ which is equivalent to $(1)$, the semigroup is a group.

\begin{lemma}
\label{lemma4.2}\hspace{2em}\\
\vspace{-4ex}
\begin{enumerate}
\item[(1)] Any principal type $a$ ideal $(\alpha)$ is equivalent to $(1+2d\sqrt{-10})$ for some $d$.
\item[(2)] $(1+2c\sqrt{-10})$ and $(1+2d\sqrt{-10})$ are equivalent if and only if $c\equiv d\bmod{2q}$. So there are $2q$ equivalence classes.
\end{enumerate}
\end{lemma}

\begin{proof}
Write $\alpha$ as $(b+2c\sqrt{-10})$. Then $(b,10)=1$ and we can choose $N$ with $(N,10)=1$ so that $Nb\equiv 1\mod{4q}$. Then $(Na)$ is of type $a$, and $N\alpha\equiv 1+2Nc\sqrt{-10}\bmod{4q}$, proving (1). Turning to (2), if $c\equiv d\bmod{2q}$ we may take $N=1$. Conversely if $N(1+2c\sqrt{-10})\equiv 1+2d\sqrt{-10}\bmod{4q}$, then $N\equiv 1\mod{4q}$. So mod $4q$, $2d\equiv 2Nc\equiv 2c$.
\qed
\end{proof}

\begin{theorem}
\label{theorem4.3}
The order $2q$ group of classes of principal type $a$ ideals is cyclic, generated by the class of any $(c+2d\sqrt{-10})$ with $(c,10)=1$ and $d$ odd.
\end{theorem}
\pagebreak

\begin{proof}
One argues as in the last few sentences of the proof of Theorem 3.2 of \cite{1}.
\qed
\end{proof}

Now fix a type $a$ non-principal ideal $L$; for example $P=(7,2-\sqrt{-10})$. If $I$ is another such ideal, $IL$ is principal of type $a$.

\begin{definition}
\label{def4.4}
Suppose $I$ and $J$ are type $a$ non-principal. $I$ and $J$ are Gauss equivalent if $IL$ and $JL$ are equivalent in the sense of Definition \ref{def4.1}.
\end{definition}

The definition appears to depend on the choice of $L$. But if one replaces $L$ by $\gamma L$ where $\chi(\mathrm{norm}(\gamma))=1$, one finds that the notion of Gauss-equivalence is unchanged. It follows at once that the dependence on $L$ is illusory. One sees further that the Gauss-classes of all type $a$ ideals form a group of order $4q$ under ideal multiplication and that the inverse of $I$ is $\bar{I}$.

\begin{theorem}
\label{theorem4.5}
The group of Gauss-classes of type $a$ ideals is cyclic of order $4q$; any class consisting of non-principal ideals is a generator.
\end{theorem}

\begin{proof}
Let $P=(7,2-\sqrt{-10})$. $P$ has norm $7$, and is of type $a$. Also, $P^{2}=(3+2\sqrt{-10})$. By Theorem \ref{theorem4.3}, $P^{2}$ has order $2q$. So $P$ has order $4q$, and the result follows.
\qed
\end{proof}

\begin{definition}
\label{def4.6}
If $R$ is a Gauss-class, $\theta(R)$ in $Z[[x]]$ is $\sum x^{\mathrm{norm}(I)}$, where $I$ runs over the ideals in $R$.
\end{definition}

\begin{definition}
\label{def4.7}\hspace{2em}\\
\vspace{-4ex}
\begin{enumerate}
\item[(1)] $e$ is the Gauss-class of (1).
\item[(2)] $\mathit{AMB}$ is the Gauss-class of order 2; that is to say the Gauss-class of $(1+2q\sqrt{-10})$.
\end{enumerate}
\end{definition}

\begin{lemma}
\label{lemma4.8}\hspace{2em}\\
\vspace{-4ex}
\begin{enumerate}
\item[(1)] The mod~2 reduction of $\theta(e)$ is $D$.
\item[(2)] $\theta(\mathit{AMB})$ is in $2Z[[x]]$, and the mod~2 reduction of $\frac{1}{2}\theta(\mathit{AMB})$ is $D_{40q^{2}+1}$.
\end{enumerate}
\end{lemma}

\begin{proof}
$I$ is in $e$ if and only if $\bar{I}$ is in $e$; also they each make the same contribution to $\theta(e)$. So in calculating $\theta(e)\bmod{2}$ we only need consider $I$ in $e$ with $\bar{I}=I$. These are just the $(N)$ with $N$ in $Z$, $(N,10)=1$ and $N>0$, and we get (1). $\mathit{AMB}$ consists of principal ideals. These are of the form $(b+2cq\sqrt{-10})$ with $(b,10)=1$, $b>0$ and $c$ odd. Since $(b+2cq\sqrt{-10})$ and $(b-2cq\sqrt{-10})$ make the same contribution to $\theta(\mathit{AMB})$, $\theta(\mathit{AMB})$ is in $2Z[[x]]$. Also, $\frac{1}{2}\theta(\mathit{AMB})$ is $\sum x^{b^{2}+40q^{2}c^{2}}$, the sum running over all positive $b$ and $c$ with $(b,10)=1$ and $c$ odd. Reducing mod~2 we get $DG^{8q^{2}}=D_{40q^{2}+1}$.
\qed
\end{proof}

We next define a Hecke operator $T_{p}:Z[[x]]\rightarrow Z[[x]]$ for each $p$ with $\chi(p)=1$.

\begin{definition}
\label{def4.9}
If $\chi(p)=1$, $T_{p}$ is the map $\sum c_{n}x^{n}\rightarrow \sum c_{pn}x^{n}+\left(\frac{-10}{p}\right)\sum c_{n}x^{pn}$.
\end{definition}

Note that the mod~2 reduction of $T_{p}$ is our usual mod~2 Hecke operator $T_{p}:Z/2[[x]]\rightarrow Z/2[[x]]$.

\begin{remark*}{Remark}
The motivation for Definition \ref{def4.9} is the following. It can be shown that each $\theta(R)$ is the expansion at infinity of a modular form of weight 1 for some $\Gamma_{1}(N)$, with character $n\rightarrow \left(\frac{-10}{n}\right)$. And Definition \ref{def4.9} is the standard definition of the Hecke action on such expansions. (But in what follows we won't use the connection of the $\theta(R)$ with weight 1 modular forms.)
\end{remark*}

\begin{theorem}
\label{theorem4.10}
Suppose $\chi(p)=1$ and $p$ is inert in $Z[\sqrt{-10}]$. Then $T_{p}$ annihilates each $\theta(R)$ in $Z[[x]]$.
\end{theorem}

\begin{proof}
$I\rightarrow pI$ sets up a 1--1 correspondence between ideals of norm $n$ in $R$ and ideals of norm $p^{2}n$ in $R$. Also, if $(n,p)=1$ there are no ideals of norm $pn$ in $R$. Since $\left(\frac{-10}{n}\right)=-1$, the result follows directly from Definition \ref{def4.9}.
\qed
\end{proof}

\begin{theorem}
\label{theorem4.11}
Suppose $\chi(p)=1$ and $p$ splits in $Z[\sqrt{-10}]$, so that $(p)=P\cdot\bar{P}$ with $P\ne \bar{P}$. Then if $R$ is a Gauss-class, $T_{p}(\theta(R))=\theta(PR)+\theta(\bar{P}R)$. (Since $\chi(p)=1$, $P$ and $\bar{P}$ are type $a$.)
\end{theorem}

\begin{proof}
The argument follows that in the proof of Theorem 3.6 of \cite{1}, the essential points being the multiplicativity of the norm and unique factorization at the ideal level in $Z[\sqrt{-10}]$.
\qed
\end{proof}

\begin{definition}
\label{def4.12}
$\alpha(R)$ is the mod~2 reduction of $\theta(R)$. $DI(q)\subset Z/2[[x]]$ is the space spanned by the $\alpha(R)$, as $R$ runs over the $4q$ Gauss-classes of type $a$ ideals.
\end{definition}

Now let $C$ be a Gauss-class containing one of the ideals of norm 7. By Theorem \ref{theorem4.5}, $C$ generates the group of Gauss-classes of type $a$ ideals.

\begin{theorem}
\label{theorem4.13}
Let $\alpha_{i}=\alpha(C^{i})$. Then $\alpha_{0}=D$, $\alpha_{2q}=0$, and the $\alpha_{i}$, $0\le i <2q$ span $DI(q)$.
\end{theorem}

\begin{proof}
Since $C$ is a generator, $\alpha_{0}, \ldots, \alpha_{4q-1}$ span $DI(q)$. By Lemma \ref{lemma4.8}, $\alpha_{0}=D$ and $\alpha_{2q}=0$. Finally $I\rightarrow\bar{I}$ sets up a 1--1 norm preserving correspondence between the ideals in $C^{i}$ and the ideals in $C^{4q-i}$, so $\alpha_{i}=\alpha_{4q-i}$.
\qed
\end{proof}

\begin{theorem}
\label{theorem4.14}
\hspace{2em}\\
\vspace{-4ex}
\begin{enumerate}
\item[(1)] $T_{7}(\alpha_{i})=\alpha_{i-1}+\alpha_{i+1}$ if $0<i<2q$. So $T_{7}$ stabilizes $DI(q)$.
\item[(2)] $T_{7}(D_{40q^{2}+1})=\alpha_{2q-1}$.
\end{enumerate}
\end{theorem}

\begin{proof}
By Theorem \ref{theorem4.11}, $T_{7}(\theta_{i})=\theta_{i-1}+\theta_{i+1}$; reducing mod~2 we get (1). Also $T_{7}(\theta(\mathit{AMB}))=\theta_{2q-1}+\theta_{2q+1}=2\theta_{2q-1}$. Dividing by 2, reducing mod~2 and using Lemma \ref{lemma4.8} we get (2).
\qed
\end{proof}

\begin{definition}
\label{def4.15}
$U_{n}$ is the element of $Z/2[t]$ with $U_{n}(t+t^{-1})=t^{n}+t^{-n}$. Note that $U_{0}=0$, $U_{1}(Y)=Y$, and that $U_{n+2}(Y)=YU_{n+1}(Y)+U_{n}(Y)$. Furthermore $U_{2n}=U_{n}^{2}$, and it follows that $U_{q}(Y)=Y^{q}$. Finally $Y$ divides each $U_{n}(Y)$.
\end{definition}

\begin{lemma}
\label{lemma4.16}
Let $Y$ be the operator $T_{7}:Z/2[[x]]\rightarrow Z/2[[x]]$. Then for $0\le i \le 2q$, $\alpha_{2q-i}=U_{i}(Y)\cdot (D_{40q^{2}+1})$. 
\end{lemma}

\begin{proof}
If $i\le 2q-2$, Theorem \ref{theorem4.14} and the recursion in Definition \ref{def4.15} give:
\begin{enumerate}
\item[(1)] $\alpha_{2q-i-2}=Y(\alpha_{2q-i-1})+\alpha_{2q-i}$.
\item[(2)] $U_{i+2}(Y)=Y\cdot U_{i+1}(Y)+U_{i}(Y)$.
\end{enumerate}
So, by induction on $i$, it suffices to show that $\alpha_{2q}$ and $\alpha_{2q-1}$ are $U_{0}(Y)(D_{40q^{2}+1})$ and $U_{1}(Y)(D_{40q^{2}+1})$.  But Lemma \ref{lemma4.8} and Theorem \ref{theorem4.14} show that $\alpha_{2q}=0$ and $\alpha_{2q-1}=Y\cdot D_{40q^{2}+1}$. 
\qed
\end{proof}

\begin{theorem}
\label{theorem4.17}
As $Z/2[Y]$-module, $DI(q)$ is cyclic with generator $\alpha_{2q-1}=Y\cdot D_{40q^{2}+1}$.
\end{theorem}

\begin{proof}
$\alpha_{2q-i}=U_{i}(Y)\cdot D_{40q^{2}+1}$ for $1\le i\le 2q$. Since $Y$ divides each $U_{i}(Y)$ we're done.
\qed
\end{proof}

\begin{theorem}
\label{theorem4.18}
$Y^{2q-1}(\alpha_{2q-1})=D$, while $Y^{2q}(\alpha_{2q-1})=0$. Also, $DI(q)$ has dimension $2q$ over $Z/2$ and is isomorphic as $Z/2[Y]$-module with $Z/2[Y]/Y^{2q}$.
\end{theorem}

\begin{proof}
$Y^{2q-1}(\alpha_{2q-1})= Y^{2q}(D_{40q^{2}+1})=U_{2q}(Y)\cdot D_{40q^{2}+1}$. By Lemma \ref{lemma4.16} this is $\alpha_{0}=D$. So $Y^{2q}(\alpha_{2q-1})= T_{7}(D)=0$. It follows that the annihilator of $\alpha_{2q-1}$ in $Z/2[Y]$ is the ideal $(Y^{2q})$. Theorem \ref{theorem4.17} then gives the final assertions.
\qed
\end{proof}

\begin{theorem}
\label{theorem4.19}
If $p$ is as in Theorem \ref{theorem4.10} then $T_{p}$ annihilates $DI(q)$. In particular, $X=T_{3}$ annihilates $DI(q)$.
\end{theorem}

\begin{proof}
This is immediate from Theorem \ref{theorem4.10}
\qed
\end{proof}

\section{The action of \bm{$T_{3}$} and \bm{$T_{7}$} on \bm{$W_{a}$}}
\label{section5}

This section is the counterpart to section 4 of \cite{1}. In that paper we defined certain subspaces $W1$ and $W5$ of $Z/2[[x]]$; see section \ref{section1} of the present paper for a summary.  In section 4 of \cite{1} we made $W5$ into a $Z/2[X,Y]$-module with $X$ and $Y$ acting by $T_{7}$ and $T_{13}$, and we showed the existence of an ``adapted basis'' $m_{i,j}$ where $D=\sum_{n>0,\ (n,6)=1}x^{n^{2}}$ and $m_{0,0}=D^{5}$. Here we'll derive a similar result with $W5$, $T_{7}$ and $T_{13}$ replaced by $W_{a}$, $T_{3}$ and $T_{7}$.

Since $\chi(3)=\chi(7)=1$, the commuting operators $T_{3}$ and $T_{7}:Z/2[[x]]\rightarrow Z/2[[x]]$ stabilize $W_{a}$; see Corollary \ref{corollary2.20}. So one can make $W_{a}$ into a $Z/2[X,Y]$-module with $X$ and $Y$ acting by  $T_{3}$ and $T_{7}$. We shall filter $W_{a}$ by finite-dimensional $Z/2[X,Y]$-stable subspaces, $W_{a}(q)$, $q$ running over the powers of $2$.

In the paragraph following Definition \ref{def2.27} we constructed an isomorphism between $W_{a}$ and a certain subspace $V^{\prime}$ of $Z/2[w]$. Using this identification we may speak of the ``$w$-degree'' of an element of $W_{a}$. Theorem \ref{theorem2.24} shows that the maps $X$ and $Y:W_{a}\rightarrow W_{a}$ lower the $w$-degree.

\begin{definition}
\label{def5.1}
$W_{a}(q)$ is the subspace of $W_{a}$, of $Z/2$-dimension $8q^{2}$, consisting of elements of $w$-degree $<40q^{2}$. (Since $X$ and $Y$ lower the $w$-degree they stabilize $W_{a}(q)$.)
\end{definition}

\begin{theorem}
\label{theorem5.2}
The space $DI(q)$ of Definition \ref{def4.12}, of $Z/2$-dimension $2q$, is a $Z/2[X,Y]$-submodule of $W_{a}(q)$ annihilated by $X$.
\end{theorem}

\begin{proof}
By Theorems \ref{theorem4.17}, \ref{theorem4.18} and \ref{theorem4.19}, $DI(q)$ has $Z/2$-dimension $2q$, is annihilated by $X$ and as $Z/2[Y]$-module is cyclic generated by $\alpha_{2q-1}=Y\cdot D_{40q^{2}+1}$. So it suffices to show that $\alpha_{2q-1}$ is in $W_{a}(q)$. But $Y$ lowers the $w$-degree.
\qed
\end{proof}

We'll use the results of section \ref{section3} to show that the kernel of $X:W_{a}(q)\rightarrow W_{a}(q)$ is precisely $DI(q)$. Note that $g(2q)=4q^{2}$. So in the language of section \ref{section3}, $\langle 0, 2q\rangle=w^{40q^{2}+1}$, and so the $w$-degree of $(0,2q)^{*}$ is $40q^{2}+1$.

\begin{lemma}
\label{lemma5.3}
Suppose $f\ne 0$ is in $W_{a}(q)$, with $Xf = 0$. Write $f$ as $(a,b)^{*}+$ a sum of earlier $D_{i}$, as in section \ref{section3}. Then $a=0$, and $0\le b< 2q$.
\end{lemma}

\begin{proof}
If $a>0$, then by Theorem \ref{theorem3.5}, $Xf =(a-1,b)^{*}+$ a sum of earlier $D_{i}$, and so cannot be $0$. So $f=(0,b)^{*}+$ a sum of earlier $D_{i}$. By Theorem \ref{theorem3.6}, the $w$-degree of $f$ is the $w$-degree of $(0,b)^{*}$. If $b\ge 2q$, then this last $w$-degree is $\ge$ the $w$-degree, $40q^{2}+1$, of $(0,2q)^{*}$. So $f$ cannot be in $W_{a}(q)$.
\qed
\end{proof}

\begin{theorem}
\label{theorem5.4}
The kernel of $X:W_{a}(q)\rightarrow W_{a}(q)$ is $DI(q)$.
\end{theorem}

\begin{proof}
By Lemma \ref{lemma5.3} this kernel has dimension $\le 2q$, and we use Theorem \ref{theorem5.2}.
\qed
\end{proof}

\begin{corollary}
\label{corollary5.5}
$DI(1)\subset DI(2)\subset DI(4)\subset \ldots$, and the kernel of $X:W_{a}\rightarrow W_{a}$ is the union, $DI$ of the $DI(q)$.
\end{corollary}

\begin{theorem}
\label{theorem5.6}
The only elements of $W_{a}$ annihilated by $X$ and $Y$ are $0$ and $D$.
\end{theorem}

\begin{proof}
If $(X,Y)f=0$, $f$ is in $DI$ by Corollary \ref{corollary5.5}, and so is in some $DI(q)$. But $DI(q)$, as $Z/2[Y]$-module, is isomorphic to $Z/2[Y]/(Y^{2q})$. It follows that the kernel of $Y:DI(q)\rightarrow DI(q)$ has dimension $1$ over $Z/2$. 
\qed
\end{proof}

\begin{definition}
\label{def5.7}
$S_{m}$ is the subspace of $W_{a}$. of dimension $m(m+1)/2$, spanned over $Z/2$ by the $(a,b)^{*}$ with $a+b<m$.
\end{definition}

Note that $S_{0}=(0)$ while $S_{1}$ is spanned by $(0,0)^{*}=D$. So $X\cdot S_{1}=Y\cdot S_{1} = S_{0}$.

\begin{lemma}
\label{lemma5.8}
$X:W_{a}\rightarrow W_{a}$ is onto. In fact $X$ maps $S_{m+1}$ onto $S_{m}$.
\end{lemma}

\begin{proof}
By Theorem \ref{theorem3.5}, $X\cdot S_{m+1}\subset S_{m}$, so it suffices to show that the kernel of $X: S_{m+1}\rightarrow S_{m}$ has dimension at most $m+1$. Suppose $f\ne 0$ is in this kernel. The proof of Lemma \ref{lemma5.3} shows that $f=(0,b)^{*}+$ a sum of earlier $D_{i}$, and that the $w$-degree of $f$ is the $w$-degree of $(0,b)^{*}$. But Theorem \ref{theorem3.6} tells us that every element of $S_{m+1}$ has $w$-degree $\le$ the $w$-degree of $(0,m)^{*}$. So $0\le b \le m$, and the result follows.
\qed
\end{proof}

\begin{theorem}
\label{theorem5.9}
$Y\cdot S_{m+1}\subset S_{m}$.
\end{theorem}

\begin{proof}
We argue by induction on $m$, $m=0$ being clear. Suppose $f$ is in $S_{m+1}$ with $m>0$. Then $Xf$ is in $S_{m}$, so by the induction hypothesis, $X(Yf)=Y(Xf)$ is in $S_{m-1}$. By Lemma \ref{lemma5.8}, there is an $h$ in $S_{m}$ with $Xh=XYf$. Then $h+Yf$, being in the kernel of $X$, is $(0,b)^{*}+$ a sum of earlier $D_{i}$ for some $b$. Since $h+Yf$ is in $S_{m+1}$, $b\le m$, and it will suffice to show that $b\ne m$. Suppose on the contrary that $h+Yf=(0,m)^{*}+$ a sum of earlier $D_{i}$. Then since $f$ is in $S_{m+1}$ and $Y$ lowers the $w$-degree, the $w$-degree of the left hand side is $<$ the $w$-degree of $(0,m)^{*}$. But that of the right hand side equals the $w$-degree of $(0,m)^{*}$, giving a contradiction.
\qed
\end{proof}

\begin{lemma}
\label{lemma5.10}
For each $m$ there is an element of $DI$ of the form $(0,m)^{*}+$ a sum of earlier $D_{i}$.
\end{lemma}

\begin{proof}
Fix $q>m$. Each $f\ne 0$ in $DI(q)$ can be written as $(0,b)^{*}+$ a sum of earlier $D_{i}$ for some $b$ with $0\le b <2q$. Since there are only $2q$ choices for $b$, and $DI(q)$ has dimension $2q$, the result follows.
\end{proof}

\begin{lemma}
\label{lemma5.11}
$DI\cap S_{m}$ has dimension $m$. Furthermore, $Y$ maps $DI\cap S_{m+1}$ onto $DI\cap S_{m}$.
\end{lemma}

\begin{proof}
By Lemma \ref{lemma5.10}, $DI\cap S_{m+1} \ne DI\cap S_{m}$. Now $Y$ maps $DI\cap S_{m+1}$ into $DI\cap S_{m}$ by Theorem \ref{theorem5.9}; Theorem \ref{theorem5.6} shows that the kernel of this map is contained in $\{0,D\}$. So the map is onto, and the dimensions of $DI\cap S_{m+1}$ and $DI\cap S_{m}$ differ by $1$.
\qed
\end{proof}

The machinery is now in place to establish the results analogous to those of section 4 of \cite{1}. The arguments are exactly the same as those that were used to derive Theorems 4.12, 4.14, 4.15, 4.16, 4.17 and Corollary 4.13 of \cite{1}.

\begin{theorem}
\label{theorem5.12}
Let $f$ and $h$ be in $S_{m}$ with $Yf=Xh$. Then there is an $e$ in $S_{m+1}$ with $Xe=f$ and $Ye=h$.
\end{theorem}

\begin{corollary}
\label{corollary5.13}
There are $m_{a,b}$ in $S_{a+b+1}$ such that:
\vspace{-2ex}
\begin{enumerate}
\item[(1)] $m_{0,0}=D$.
\item[(2)] $X\cdot m_{a,b}= m_{a-1,b}$ or $0$ according as $a>0$ or $a=0$.
\item[(3)] $Y\cdot m_{a,b}= m_{a,b-1}$ or $0$ according as $b>0$ or $b=0$.
\item[(4)] The $m_{a,b}$ are a $Z/2$-basis of $W_{a}$.
\end{enumerate}
\end{corollary}

\begin{theorem}
\label{theorem5.14}
Make $W_{a}$ into a $Z/2[[X,Y]]$-module with $X$ and $Y$ acting by $T_{3}$ and $T_{7}$. (This is possible since $T_{3}$ and $T_{7}$ lower the $w$-degree.) Then the action of $Z/2[[X,Y]]$ on $W_{a}$ is faithful.
\end{theorem}

\begin{theorem}
\label{theorem5.15}
If $\chi(p)=1$, $T_{p}:W_{a}\rightarrow W_{a}$ is multiplication by some $u$ in the ideal $(X,Y)$ of $Z/2[[X,Y]]$. In other words, $T_{p}$ in its action on $W_{a}$ is a power series with $0$ constant term in $T_{3}$ and $T_{7}$.
\end{theorem}

We next prove an analogue to Theorem 4.18 of \cite{1}.  Since $\chi(11)=-1$, $T_{11}(W_{a})\subset W_{b}$, $T_{11}(W_{b})\subset W_{a}$, and $T_{11}^{2}$ stabilizes $W_{a}$. We now decompose $W_{a}$ into a direct sum of 4 summands. The first is spanned by the $D_{k}$ with $k\equiv 1$ or $9\mod{40}$. For the second, $k\equiv 3$ or $27$, for the third, $k\equiv 7$ or $23$, and for the fourth, $k\equiv 21$ or $29$. By Lemma \ref{lemma2.10}, the exponents appearing in elements of the first summand are $1$ or $9\bmod{40}$, and corresponding results hold for the other summands.

\begin{theorem}
\label{theorem5.16}
$T_{11}^{2}:W_{a}\rightarrow W_{a}$ is multiplication by $\lambda^{2}$ for some $\lambda$ in the ideal $(X,Y)$ of $Z/2[[X,Y]]$. 
\end{theorem}

\begin{proof}
As in the proofs of Theorem 4.16 and 4.17 of \cite{1} we use the $Z/2[[X,Y]]$-linearity of our map to show that it is multiplication by some $u$ in $(X,Y)$ and we write $u$ as $a+bX+cY+dXY$ with $a,b,c,d$ in $Z/2[[X^{2},Y^{2}]]$. Let $a=\lambda^{2}$. We'll show that for each $h$ in $W_{a}$, $T_{11}^{2}(h)=\lambda^{2}h$. We may assume that $h$ is in one of the 4 subspaces of the above direct sum decomposition. Suppose for example that it is in the second. Then $T_{11}^{2}(h)$ and $ah$ are sums of $D_{k}$, $k\equiv 3$ or $27\mod{40}$, $(bX)h$ is a sum of $D_{k}$, $k\equiv 1$ or $9\mod{40}$, $(cY)h$ is a sum of $D_{k}$, $k\equiv 21$ or $29\mod{40}$ and $(dXY)h$ is a sum of $D_{k}$, $k\equiv 7$ or $23\mod{40}$.  Since $T_{11}^{2}(h) = uh$ is the sum of $\lambda^{2}h$, $(bX)h$, $(cY)h$ and $(dXY)h$, and the decomposition is direct, the result follows.
\qed
\end{proof}

\begin{lemma}
\label{lemma5.17}
Write the $\lambda$ of Theorem \ref{theorem5.16} as $cX+dY+\cdots$, with $c$ and $d$ in $\{0,1\}$. Then $c=d=1$.
\end{lemma}
\pagebreak

\begin{proof}
On $W_{a}$, $T_{11}^{2}=\lambda^{2}=cX^{2}+dY^{2}+\cdots$. Applying both sides to $D_{9}$ we find that $0=cD_{1}+dD_{1}$, while applying both sides to $D_{41}$ shows that $D_{1}=dD_{1}$.
\qed
\end{proof}

\begin{theorem}
\label{theorem5.18}
$T_{11}^{2}$ maps $S_{m+2}$ onto $S_{m}$ for all $m$. So $T_{11}^{2}:W_{a}\rightarrow W_{a}$ is onto.
\end{theorem}

\begin{proof}
Using the explicit description of the action of $T_{3}$ and $T_{7}$ on $W_{a}$ provided by Corollary \ref{corollary5.13} we see that $X^{2}+Y^{2}$ maps $S_{m+2}$ onto $S_{m}$. Since $T_{11}^{2}=X^{2}+Y^{2}+\cdots$, $T_{11}^{2}$ and $X^{2}+Y^{2}$ induce the same map $S_{m+3}/S_{m+2}\rightarrow S_{m+1}/S_{m}$. So this map is onto for all $m$, and an induction gives the result.
\qed
\end{proof}

\section{\bm{$T_{11}:W_{b}\rightarrow W_{a}$} is bijective}
\label{section6}
Since $\chi(11)=-1$, $T_{11}$ maps $W_{a}$ to $W_{b}$ and $W_{b}$ to $W_{a}$. In this section we show that $T_{11}(D_{k})$ is a sum of $D_{i}$, $i<k$, and that $T_{11}:W_{b}\rightarrow W_{a}$ is bijective. Let $C_{k}$, $(k,10)=1$, be $T_{11}(D_{k})$.

\begin{lemma}
\label{lemma6.1}
If $k<120$ then:
\vspace{-2ex}
\begin{enumerate}
\item[(a)] When $k\equiv 1,3,7$ or $9\mod{20}$, $C_{k}$ is a sum of $D_{i}$, $i<k$.
\item[(b)] When $k\equiv 11\ (\mbox{resp.}\ 19,13,17) \bmod{20}$, $C_{k}=D_{j}+$ a sum of $D_{i}$, $i<j$, where $j=k-10\ (\mbox{resp.}\ k-10,k-6,k-14)$.
\end{enumerate}
\end{lemma}

\begin{proofsketch}
Using Corollary \ref{corollary2.20} and Lemmas \ref{lemma2.22} and \ref{lemma2.23} we see that it suffices to prove (a) when $k\equiv 9\mod{20}$, i.e.\ when $k$ is in $\{9,29,49,69,89,109\}$. In fact the corresponding $C_{k}$ are $0$, $0$, $D_{19}+D_{11}$, $D_{39}+D_{31}$, $D_{11}$ and $D_{39}$. To calculate $C_{109}$, for example, we use Corollary \ref{corollary2.20} and Lemma \ref{lemma2.23} to see that it is a sum of $D_{i}$ with each $i$ in $\{31,39,71,79,111\}$. Examining the expansions of $D_{109}$ and the $D_{i}$ and arguing as in the proof of Theorem \ref{theorem2.26} we find that $C_{109}=D_{39}$. (b) follows from the more precise results:
\begin{enumerate}
\addtolength{\itemindent}{.2em}
\item[(b1)] $C_{11}$, $C_{31}$, $C_{51}$, $C_{71}$, $C_{91}$ and $C_{111}$ are $D_{1}$, $D_{21}$, $D_{41}+D_{9}$, $D_{61}+D_{29}$, $D_{81}+D_{41}+D_{9}$, $D_{101}+D_{61}+D_{21}$.

\item[(b2)] $C_{19}$, $C_{39}$, $C_{59}$, $C_{79}$, $C_{99}$ and $C_{119}$ are $D_{9}$, $D_{29}$, $D_{49}+D_{9}$, $D_{69}+D_{29}$, $D_{89}+D_{49}+D_{9}$, $D_{109}+D_{69}+D_{29}+D_{21}$.

\item[(b3)] $C_{13}$, $C_{33}$, $C_{53}$, $C_{73}$, $C_{93}$ and $C_{113}$ are $D_{7}$, $D_{27}+D_{3}$, $D_{47}$, $D_{67}+D_{43}+D_{27}$, $D_{87}+D_{47}$, $D_{107}+D_{83}+D_{67}+D_{43}+D_{27}$.

\item[(b4)] $C_{17}$, $C_{37}$, $C_{57}$, $C_{77}$, $C_{97}$ and $C_{117}$ are $D_{3}$, $D_{23}+D_{7}$, $D_{43}$, $D_{63}+D_{47}+D_{23}+D_{7}$, $D_{83}+D_{43}$, $D_{103}+D_{87}+D_{63}+D_{47}+D_{23}$.
\end{enumerate}
\vspace{-4ex}
\hspace{2em}\\
To calculate $C_{93}$, for example, we use Corollary \ref{corollary2.20} and Lemma \ref{lemma2.22} to see that it is a sum of $D_{i}$ where each $i$ is in $\{7,23,47,63,87\}$. Examining the expansions of $D_{93}$ and the $D_{i}$ and arguing as in the proof of Theorem \ref{theorem2.26} we find that $C_{93}=D_{87}+D_{47}$.
\qed
\end{proofsketch}

\begin{lemma}
\label{lemma6.2}
For $u$ in $Z/2[[x]]$, $T_{11}(uG^{24})$ is the sum of $G^{2i}T_{11}(uG^{2j})$ where $(i,j)$ runs over the 9 pairs $(12,0)$, $(8,4)$, $(4,8)$, $(6,2)$, $(2,6)$, $(9,1)$, $(1,9)$, $(3,3)$, $(1,1)$.
\end{lemma}

\begin{proof}
We argue as in the proof of Lemma \ref{lemma2.25}. Let $U$ be the 2 variable polynomial $(A+B)^{12}+A^{6}B^{2}+A^{2}B^{6}+A^{9}B+AB^{9}+A^{3}B^{3}+AB$. Then $U(F(x^{11}), F(x))=0$; this is the level 11 modular equation for $F$. Replacing $x$ by $x^{5}$ we find that $U(G(x^{11}), G(x))=0$. Now let $L$ be an algebraic closure of $Z/2$. We have 12 imbeddings $\varphi_{k}:Z/2[[x]]\rightarrow L[[x^{1/11}]]$, the first of which takes $f$ to $f(x^{11})$, while each of the others takes $f$ to $f(\lambda x^{1/11})$ for some $\lambda$ in $L$ with $\lambda^{11}=1$. Replacing $x$ by $\lambda x^{1/11}$ in the identity $U(G(x^{11}), G(x))=0$ and using the symmetry of $U$ we find that each $U(\varphi_{k}(G),G)$ is $0$. Squaring and expanding we find that $\varphi_{k}(G^{24})$ is the sum of the $G^{2i}\varphi_{k}(G^{2j})$ where $(i,j)$ runs over the 9 pairs above. Now the definition of $T_{11}$ shows that if $f$ is in $Z/2[[x]]$, $T_{11}(f)$ is the sum of the $\varphi_{k}(f)$. Multiplying the $k$\textsuperscript{th} of our identities by $\varphi_{k}(u)$ and summing we get the result.
\qed
\end{proof}

\begin{lemma}
\label{lemma6.3}
The conclusions (a) and (b) of Lemma \ref{lemma6.1} hold for all $k$ with $(k,10)=1$.
\end{lemma}

\begin{proof}
We argue by induction on $k$. For $k<120$, Lemma \ref{lemma6.1} applies. Suppose $k=n+120$ with $n>0$. By Lemma \ref{lemma6.2}, $C_{n+120}$ is a sum of 9 terms each corresponding to one of the 9 pairs $(i,j)$ of the lemma. These terms are $G^{24}C_{n}$, $G^{16}C_{n+40}$, $G^{8}C_{n+80}$, $G^{12}C_{n+20}$, $G^{4}C_{n+60}$, $G^{18}C_{n+10}$, $G^{2}C_{n+90}$, $G^{6}C_{n+30}$ and $G^{2}C_{n+10}$. The induction hypothesis shows that the term corresponding to $(i,j)$ is a sum of $D_{m}$ with $m<(n+10j)+10i$. Since $i+j\le 12$, each $m<n+120$, and in particular (a) holds for $k$. We turn to (b). For each of the last 6 terms in our sum, $i+j\le 10$. So the corresponding term is a sum of $D_{m}$ with $m<n+100$. Now consider the first 3 terms. Suppose for example that $n\equiv 17\mod{20}$. By the induction hypothesis, each of $G^{24}C_{n}$, $G^{16}C_{n+40}$ and $G^{8}C_{n+80}$ is $D_{n+106}+$ a sum of $D_{m}$ with $m<n+106$. It follows that $C_{n+120}$ itself is $D_{n+106}+$ a sum of $D_{m}$ with $m<n+106$. The proof of (b) when $n\equiv 11,19$ or $13\mod{20}$ is identical.
\qed
\end{proof}

\begin{theorem}
\label{theorem6.4}
If $(k,10)=1$, $T_{11}(D_{k})$ is a sum of $D_{i}$ with $i<k$. Furthermore, $T_{11}:W_{b}\rightarrow W_{a}$ is bijective.
\end{theorem}

\begin{proof}
Lemma \ref{lemma6.3} gives the first result. Also the $D_{k}$ where $k>0$ and $\equiv 11,19,13$ or $17\mod{20}$ form a basis of $W_{b}$, while the corresponding $D_{k-10}$, $D_{k-10}$, $D_{k-6}$ and $D_{k-14}$ form a basis of $W_{a}$. So Lemma \ref{lemma6.3} gives the second result as well.
\qed
\end{proof}

\begin{corollary}
\label{corollary6.5}
$T_{11}:W_{a}\rightarrow W_{b}$ is onto.
\end{corollary}

\begin{proof}
By Theorem \ref{theorem5.18}, $T_{11}$ maps the subspace $T_{11}(W_{a})$ of $W_{b}$ onto $W_{a}$, and we use the last theorem to see that $T_{11}(W_{a})$ is all of $W_{b}$.
\qed
\end{proof}

\begin{corollary}
\label{corollary6.6}
Let $\lambda$ be as in Theorem \ref{theorem5.16}. Then in its action on $W=W_{a}\oplus W_{b}$, $T_{11}^{2}$ is multiplication by $\lambda^{2}$. (This follows from Corollary \ref{corollary6.5} and Theorem \ref{theorem5.16}.)
\end{corollary}

\section{The algebra \bm{$\newo$} acting on \bm{$W_{a}$} }
\label{section7}

Take $\lambda$ in $(X,Y)$ as in Theorem \ref{theorem5.16}, and let $U:W\rightarrow W$ be the map $h\rightarrow \lambda(X,Y)h+T_{11}(h)$. Since $U$ is $Z/2[[X,Y]]$-linear, and $U^{2}$ annihilates $W$ (by Corollary \ref{corollary6.6}), $W$ has the structure of $Z/2[[X,Y]][\varepsilon]$ module with $\varepsilon^{2}=0$, $\varepsilon$ acting by $U$, and $X$ and $Y$ by $T_{3}$ and $T_{7}$. Let $\newo$ be the local ring  $Z/2[[X,Y]][\varepsilon]$.

\begin{lemma}
\label{lemma7.1}
The only element of $W_{b}$ annihilated by $\varepsilon$ is $0$.
\end{lemma}

\begin{proof}
Suppose $\varepsilon h_{b} =0$, with $h_{b}$ in $W_{b}$. Then $T_{11}(h_{b})=\lambda(T_{3},T_{7})h_{b}$. Since the first of these is in $W_{a}$ and the second in $W_{b}$, $T_{11}(h_{b})=0$. By Theorem \ref{theorem6.4}, $h_{b}=0$.
\qed
\end{proof}

In Corollary \ref{corollary5.13} we constructed a $Z/2$-basis $m_{i,j}$ of $W_{a}$ ``adapted to $X=T_{3}$ and $Y=T_{7}$.''  Since $T_{11}:W_{b}\rightarrow W_{a}$ is bijective, this basis pulls back under $T_{11}$ to a basis $n_{i,j}$ of $W_{b}$ ``adapted to $T_{3}$ and $T_{7}$.'' 

\begin{theorem}
\label{theorem7.2}\hspace{2em}\\
\vspace{-2ex}
\begin{enumerate}
\item[(1)] The $n_{i,j}$ and the $\varepsilon n_{i,j}$ form a $Z/2$-basis of $W$.
\item[(2)] $\newo$ acts faithfully on $W$.
\item[(3)] Each $T_{p}$, $p\ne 2$ or $5$, acts on $W$ by multiplication by some $r+t\varepsilon$ with $r$ and $t$ in $Z/2[[X,Y]]$.
\end{enumerate}
\end{theorem}

\begin{proof}
The arguments are just like those giving Theorems 5.2, 5.3 and 5.4 of \cite{1}, using the basis $n_{i,j}$ of $W_{b}$.
\qed
\end{proof}

\begin{theorem}
\label{theorem7.3}\hspace{2em}\\
\vspace{-2ex}
\begin{enumerate}
\item[(1)] If $\chi(p)=1$, $T_{p}:W\rightarrow W$ is multiplication by some $t$ in the maximal ideal $(X,Y)$ of $Z/2[[X,Y]]$.
\item[(2)] If $\chi(p)=-1$, $T_{p}:W\rightarrow W$ is the composition of $T_{11}$ with multiplication by some $t$ in the maximal ideal $(X,Y)$ of $Z/2[[X,Y]]$.
\end{enumerate}
\end{theorem}

\begin{proof}
Suppose $\chi(p)=1$. By Theorem \ref{theorem5.15}, $T_{p}:W_{a}\rightarrow W_{a}$ is multiplication by some $t$ in the maximal ideal $(X,Y)$. Since $T_{11}:W_{a}\rightarrow W_{b}$ is onto, $T_{p}:W_{b}\rightarrow W_{b}$ is multiplication by the same $t$, and we get (1). Now $T_{11}$ is multiplication by $\lambda+\varepsilon$. Suppose $\chi(p)=-1$. By Theorem \ref{theorem7.2}, $T_{p}$ is multiplication by some $r+t\varepsilon$ with $r$ and $t$ in $Z/2[[X,Y]]$.  Then $T_{p}+tT_{11}$ is multiplication by $r+\lambda t$. Since $T_{p}+tT_{11}$ and multiplication by $r+\lambda t$ map $W_{b}$ into $W_{a}$ and $W_{b}$ respectively, $T_{p}+tT_{11}=0$, giving (2).
\qed
\end{proof}

\begin{corollary}
\label{corollary7.4}
If $p\ne 2$ or $5$ and $(k,10)=1$, then $T_{p}(D_{k})$ is a sum of $D_{i}$ with $i<k$.
\end{corollary}

\begin{proof}
We have seen that this holds when $p=3$, $7$ or $11$; Theorem \ref{theorem7.3} then gives the general result.
\qed
\end{proof}

Theorems \ref{theorem7.2} (2) and \ref{theorem7.3} tell us that when we complete the Hecke algebra generated by the $T_{p}$ acting on $W$ with respect to the maximal ideal generated by the $T_{p}$, then the completed Hecke algebra we get is just the (non-reduced) local ring $\newo$. This is completely analogous to the results of \cite{1}; see Theorems 5.3 and 5.5 of that paper.



\end{document}